\documentclass[11pt,reqno]{amsart}
\usepackage{amscd}
\usepackage{color}
\usepackage[symbol]{footmisc}
\usepackage{amssymb}
\usepackage{amsfonts}
\usepackage{latexsym}
\usepackage{verbatim}
\usepackage{bbm}

\newcommand{\R}{\mathbb{R}}
\newcommand{\C}{\mathbb{C}}
\renewcommand{\H}{\mathbb{H}}
\renewcommand{\epsilon}{\varepsilon}
\renewcommand{\theta}{\vartheta}
\renewcommand{\phi}{\varphi}
\DeclareMathOperator*{\osc}{osc}
\renewcommand{\Re}{\mathrm{Re}}

\newcommand\norma[1]{\left\lVert#1\right\rVert}

\theoremstyle{plain}
\newtheorem{theorem}{Theorem}
\newtheorem{proposition}[theorem]{Proposition}
\newtheorem{lemma}[theorem]{Lemma}

\numberwithin{equation}{section}

\theoremstyle{definition}
\newtheorem*{remark}{Remark}

\theoremstyle{definition}
\newtheorem{definition}{Definition}

\title[Fully non-linear parabolic equations on compact hyperhermitian manifolds]{Fully non-linear parabolic equations on compact manifolds with a flat hyperk\"ahler metric}

\begin{document}
	
	
	\address{Dipartimento di Matematica G. Peano \\ Universit\`a di Torino\\
		Via Carlo Alberto 10\\	10123 Torino\\ Italy.}
	\email{giovanni.gentili@unito.it}
	
	\address{Jiaogen Zhang, School of Mathematical Sciences, University of Science and Technology of China, Hefei 230026, People's Republic of China.}
	\email{zjgmath@ustc.edu.cn}
	
	
	\author{Giovanni Gentili and Jiaogen Zhang}
	
	\date{\today}

	\maketitle
	\begin{abstract}
		Our recent work about fully non-linear elliptic equations on compact manifolds with a flat hyperk\"ahler metric is hereby extended to the parabolic setting. This approach will help us to study some problems arising from hyperhermitian geometry.
	\end{abstract}
	
	\tableofcontents
	
	\section{Introduction}
	After Yau's solution \cite{Y} of the Calabi conjecture \cite{Calabi}, Cao \cite{Cao} was able to provide a parabolic proof, using what is now called the K\"ahler-Ricci flow. Ever since then, it is now a well-established practice to design parabolic geometric flows as an alternative way to solve fully non-linear elliptic equations (see e.g. \cite{BGV,Chu2,FL,FLM,Gill,Gill2,SW19,SW,Sun15,Sun,Sun19,Zhang,Z,Zheng}).
	
	Following this line of thoughts, we extend to the parabolic setting the investigation of our previous work \cite{GZ}, in which we studied a class of fully non-linear elliptic equations on hyperhermitian manifolds. More precisely we took into account equations that are symmetric in the eigenvalues of the quaternionic Hessian of the unknown. In \cite{GZ} we were inspired by the work of Sz\'{e}kelyhidi \cite{Sze}, while here we develop the corresponding parabolic theory in the same spirit as Phong-T\^o \cite{Phong-To}. 
	\medskip
	
	Let $(M,I,J,K,g,\Omega_0)$ be a compact locally flat hyperhermitian manifold where $\Omega_0$ is the $(2,0)$-form induced by $g$, i.e. $\Omega_{0}=g(J\cdot,\cdot)+ig(K\cdot,\cdot)$ (see Section \ref{pre} for further details on the relevant definitions). Here, and throughout the paper, the type decomposition is taken with respect to $I$. The assumption of local flatness allows us to represent locally in quaternionic coordinates every q-real $(2,0)$-form $\Omega$ by a hyperhermitian matrix $ (\Omega_{\bar r s})$. The same is true for a hyperhermitian metric $g$, whose corresponding matrix we denote $(g_{\bar rs})$. Fix one such form $\Omega$, which does not need to coincide with $\Omega_0$. 
	Consider the operator $\partial_J:=J^{-1}\bar \partial J$, where $\bar \partial$ (as well as $\partial$) is taken with respect to $I$ everywhere in the paper. For a smooth real function $\varphi$ on $M$ the $(2,0)$-form $\partial \partial_J \phi$ is q-real.  Then we may associate a hyperhermitian matrix to the form
	\[
	\Omega_{\varphi}:=\Omega+\partial\partial_J\varphi
	\]
	let us denote it by $(\Omega^\phi_{\bar r s})$. Set $ A^r_s[\varphi]=g^{\bar j r}\Omega^\phi_{ \bar j s}$, where $(g^{\bar j r})$ is the inverse matrix of $(g_{\bar j r})$.  
	The matrix $(A^r_s[\varphi])$ defines a hyperhermitian endomorphism of $ TM $ with respect to the metric $ g $ and this makes it meaningful to speak about the $ n $-tuple of its eigenvalues $ \lambda(A[\varphi]) $.
	
	The class of parabolic equations that we take into account in the present paper is the following: 
	\begin{equation}
		\label{eq_main}
		\partial_{t}\varphi=F(A[\varphi])-h\,, \qquad \phi(x,0)=\phi_0\,, \qquad t\in [0,\infty)\,,
	\end{equation}
	where $ h\in C^\infty(M,\R) $ is the datum and $ F(A[\phi])=f(\lambda(A[\phi])) $ is a smooth symmetric operator of the eigenvalues of $ A[\phi] $ satisfying certain assumptions. More precisely, let $ \Gamma $ be a proper convex open cone in $ \R^n $ with vertex at the origin, containing the positive orthant
	\begin{equation*}
		\Gamma_n=\{ \lambda=(\lambda_1,\dots,\lambda_n) \in \R^n \mid \lambda_i>0,\, i=1,\dots,n \}\,,
	\end{equation*}
	and assume that $ \Gamma $ is symmetric, i.e. it is invariant under permutations of the $ \lambda_i $'s. We require that $ f\colon \Gamma \to \R $ satisfies the following assumptions:
	\begin{enumerate}
		\itemsep0.2em
		\item[C1)] $ f_i:=\frac{\partial f}{\partial \lambda_i}>0 $ for all $ i=1,\dots,n $ and $ f $ is a concave function.
		\item[C2)] $ \sup_{\partial \Gamma}f<\inf_Mh $, where $ \sup_{\partial \Gamma}f=\sup_{\lambda_0\in \partial \Gamma} \limsup_{\lambda \to \lambda_0} f(\lambda) $.
		\item[C3)] For any $ \sigma <\sup_\Gamma f $ and $ \lambda \in \Gamma $ we have $ \lim_{t\to \infty} f(t\lambda)>\sigma $.
	\end{enumerate}
	Assumption C1 implies parabolicity of equation \eqref{eq_main} over the space of $\Gamma$-admissible functions, where a function $\phi \in C^{1,1}(M\times [0,T))$ is {\em $\Gamma$-admissible} if
	\[
	\lambda( A[\varphi] )\in \Gamma\,,\qquad \text{for all }(x,t)\in M\times [0,T)\,.
	\]
	In particular, from standard parabolic theory, equation \eqref{eq_main} admits a unique maximal smooth solution. Assumption C2 guarantees that the level sets of $f$ do not intersect the boundary of $\Gamma$, this yields non-degeneracy of \eqref{eq_main} and entails uniform parabolicity, once we obtain the $C^{1,1}$ estimate. We also remark that the assumptions on $\Gamma$ imply the inclusion
	\begin{equation}\label{Gamma in Gamma1}
		\Gamma \subseteq \left\{ (\lambda_1,\dots, \lambda_n) \in \R^n \mid \sum_{i=1}^n \lambda_i>0 \right\}\,.
	\end{equation}
	\medskip
	
	We now project $\Gamma$ onto a new cone in $\R^{n-1}$:
	\[
	\Gamma_{\infty}=\{\lambda'=(\lambda_{1},\cdots,\lambda_{n-1})\in \mathbb{R}^{n-1} \mid \text{there exists} \ \lambda_{n}\in \mathbb{R} \ \text{such that} \ (\lambda',\lambda_{n})\in \Gamma\}\,.
	\]
	Therefore, for every $\lambda'\in \Gamma_\infty$, there exists a constant $s_{0}$ such that for each $s\geq s_{0}$, we have $(\lambda',s)\in \Gamma$. Let $ f_{\infty}(\lambda')=\lim_{s\rightarrow \infty}f(\lambda',s).$ It is an observation of Trudinger \cite{Trudinger} that, since $f$ is concave on $\Gamma$, there is a dichotomy:
	\begin{enumerate}
		\itemsep0.2em
		\item[(i)] Either $f_{\infty}$ is unbounded at any point in $\Gamma_{\infty}$ and we will refer to this case by saying that $f$ is \textit{unbounded} over $\Gamma$;
		\item[(ii)] Or $f_{\infty}$ is bounded on $\Gamma_{\infty}$ and we will simply say that $f$ is \textit{bounded} over $\Gamma$.
	\end{enumerate}

	Before stating our main results, we need to recall the terminology of parabolic $\mathcal{C}$-subsolutions introduced in \cite{Phong-To}.
	\begin{definition}\label{Def}
		We say that a function $\underline{\phi}\in C^{1,1}(M\times [0,T))$ is a {\em parabolic $\mathcal{C}$-subsolution} for equation \eqref{eq_main} if there exist uniform constants $\delta, R>0$, such that on $M\times [0,T)$,
		\begin{equation}\label{cone condition}
			f(\lambda(A[\underline{\phi}])+\mu)-\partial_{t} \underline{\phi}+\tau=h, ~ \mu+\delta \textbf{1}\in \Gamma_{n}~\textrm{and}~\tau>-\delta
		\end{equation}
		implies that $|\mu|+|\tau|<R $, where ${\bf 1}=(1,1,\dots,1)$.
	\end{definition}
	
	In the unbounded case, as we shall show, any $\Gamma$-admissible function is a parabolic $\mathcal{C}$-subsolution, and we have the following result:
	\begin{theorem}\label{unbounded theorem}
		Suppose $f$ is unbounded on $\Gamma$. Let $(M,I,J,K,g,\Omega_0)$ be a compact flat hyperk\"ahler manifold. Then for any $\Gamma$-admissible initial datum $\varphi_{0}$, the solution $\varphi$ to \eqref{eq_main} exists for all time.
		
		Moreover, if we let
		\begin{equation}\label{tilde varphi}
			\tilde{\varphi}=\varphi-\frac{\int_{M}\varphi\,\Omega^{n}_0\wedge\bar{\Omega}^{n}_0}{\int_{M}\Omega^{n}_0\wedge\bar{\Omega}^{n}_0}\,,
		\end{equation}
		then $\tilde{\varphi}$ converges smoothly to some function $\tilde{\varphi}_{\infty}\in C^\infty(M,\R)$ as $t\rightarrow\infty$, and there exists a constant $b\in\mathbb{R}$  such that
		\begin{equation}\label{elliptic function}
			F(A[\tilde{\varphi}_{\infty}])=h+b\,.
		\end{equation}  
	\end{theorem}
	\medskip
	
	Before we discuss the bounded case we present two of the many possible applications provided by Theorem \ref{unbounded theorem}, namely we show the convergence of the {\em quaternionic Hessian flow} and of the {\em $(n-1)$-quaternionic plurisubharmonic flow} on compact flat hyperk\"ahler manifolds. 
	
	Let $(M,I,J,K,g,\Omega_0)$ be a compact hyperhermitian manifold. Let $ 1\leq k \leq n $ and fix a q-real $k$-positive $(2,0)$-form $\Omega$, that is
	\begin{equation*}
		\frac{\Omega^{i}\wedge \Omega_0^{n-i}}{\Omega_0^n}>0 \qquad \mbox{ for every }i=1,\dots,k\,. 
	\end{equation*}
	Then the {\em quaternionic Hessian flow} can be written as 
	\begin{equation}\label{qh}
		\partial_{t}\varphi=\log \frac{\Omega^{k}_\varphi \wedge \Omega^{n-k}_0}{\Omega^{n}_0}-H\,,\qquad \varphi\in {\rm QSH}_{k}(M,\Omega)\,,
	\end{equation}
	where $ H\in C^\infty(M,\R) $ is the datum and ${\rm QSH}_{k}(M,\Omega)$ is the space of continuous functions $\varphi$ such that $\Omega_{\varphi}$ is a $k$-positive $q$-real $(2,0)$-form in the sense of currents.
	
	\begin{theorem}\label{theorem Hessian}
		Let $(M,I,J,K,g,\Omega_0)$ be a compact flat hyperk\"ahler manifold and $ \Omega $ a q-real $ k $-positive $ (2,0) $-form. Then for any smooth initial datum $\varphi_{0}\in {\rm QSH}_{k}(M,\Omega)$,
		\begin{enumerate}
			\itemsep0.2em
			\item the solution $\varphi$ to \eqref{qh} exists for all time;
			\item the normalization $\tilde\varphi$ (defined as in \eqref{tilde varphi}) converges smoothly as $t\rightarrow\infty$ to a function $\tilde \varphi_{\infty}\in {\rm QSH}_{k}(M,\Omega)$ , and there exists a constant $b\in\mathbb{R}$  such that
			\begin{equation}\label{qh-1}
				\frac{\Omega^{k}_{\tilde \varphi_{\infty}} \wedge \Omega^{n-k}_0}{\Omega^{n}_0}=b\,\mathrm{e}^H\,.  
			\end{equation}
		\end{enumerate}
	\end{theorem}
	We remark that the constant $b$ in \eqref{qh-1} is uniquely determined by
	\begin{equation*}
		b=\frac{\int_M \Omega^k_{\tilde \phi_{\infty}}\wedge \Omega^{n-k}_0\wedge \bar \Omega_0^n}{\int_M \mathrm{e}^H \Omega^n_0\wedge \bar \Omega_0^n}\,.
	\end{equation*}
	
	Flow \eqref{qh} provides the quaternionic counterpart of the complex Hessian flow (see e.g. \cite{SW19}). For $k=1$ equation \eqref{qh} is the parabolic Poisson equation, while for $k=n$ it becomes the parabolic quaternionic Monge-Ampère equation. Thus, Theorem \ref{theorem Hessian} generalizes the main result of \cite{BGV,Z}, which was inspired by the investigation of the quaternionic Monge-Amp\`{e}re equation on compact HKT manifolds proposed in \cite{Alesker-Verbitsky (2010)} as an analogue of the Calabi-Yau Theorem. Broadly speaking, HKT geometry constitutes a promising and interesting quaternionic analogue of K\"ahler geometry (see e.g. \cite{Alesker-Verbitsky (2006),Banos,Barberis-Fino,Chen,GF,GLV,Grantcharov-Poon (2000),Howe-Papadopoulos (1996),Ivanov,Lejmi-Weber,SThesis,Swann,Verbitsky (2002),Verbitsky (2007)} and the reference therein). As shown in \cite{Verbitsky (2009)}, the solvability of such a ``quaternionic Calabi conjecture'' would lead to an interesting geometric application as it would imply the existence of a balanced HKT metric on compact HKT manifolds with holomorphically trivial canonical bundle. Despite the equation not yet being entirely solved, there are some partial results available (see \cite{Alesker (2013),Alesker-Shelukhin (2013),Alesker-Shelukhin (2017),BGV,DinewSroka,GentiliVezzoni,GV,GZ,Sroka,Z}). 
	\medskip	
	
	Our second aforementioned application is the $(n-1)$-quaternionic plurisubharmonic flow. Let $(M,I,J,K,g,\Omega_0)$ be a compact hyperhermitian manifold and $\Omega_{1}$ be a positive $q$-real $(2,0)$-form. Denote with $\Delta_g $ the quaternionic Laplacian with respect to $g$. The {\em $(n-1)$-quaternionic plurisubharmonic flow} is encoded in the following parabolic equation: 	
	\begin{equation}\label{n-1 qpsh flow}
		\partial_t \phi=\log\frac{ \left(\Omega_{1}+\frac{1}{n-1}\big[(\Delta_{g} \phi)\Omega_{0}-\partial\partial_J \phi\big]\right)^{n}}{\Omega_0^n}-H\,, \qquad \phi \in {\rm QPSH}_{n-1}(M,\Omega_1,\Omega_0)\,,		
	\end{equation}
	where ${\rm QPSH}_{n-1}(M,\Omega_1,\Omega_0)$ denotes the space of continuous functions $\phi$ that are $(n-1)$-quaternionic plurisubharmonic with respect to $\Omega_1$ and $\Omega_0$, i.e. $ \Omega_{1}+\frac{1}{n-1}\big[(\Delta_{g} \phi)\Omega_{0}-\partial\partial_J \phi \big]>0 $ in the sense of currents.	
	\begin{theorem}\label{theorem n-1 qpsh flow}
		Let $(M,I,J,K,g,\Omega_0)$ be a compact flat hyperk\"ahler manifold and $ \Omega_{1} $ a q-real positive $ (2,0) $-form. Then for any smooth initial datum $\phi_0\in {\rm QPSH}_{n-1}(M,\Omega_1,\Omega_0)$,
		\begin{enumerate}
			\itemsep0.2em
			\item the solution $\phi$ to \eqref{n-1 qpsh flow} exists for all time;
			\item the normalization $\tilde\varphi$ of $\phi$ (defined as in \eqref{tilde varphi}) converges smoothly as $t\rightarrow\infty$ to a function $\tilde \varphi_{\infty}\in {\rm QPSH}_{n-1}(M,\Omega_1,\Omega_0)$, and there exists a constant $b\in\mathbb{R}$  such that
			\begin{equation}\label{n-1 qpsh flow 1}
				\left(\Omega_{1}+\frac{1}{n-1}\big[(\Delta_{g} \tilde \phi_\infty)\Omega_{0}-\partial\partial_J \tilde \phi_\infty \big]\right)^{n}=b\, {\rm e}^H\Omega_0^n\,.
			\end{equation}
		\end{enumerate}
	\end{theorem}
	The constant $ b $ in \eqref{n-1 qpsh flow 1} is uniquely determined by
	\begin{equation*}
		b=\frac{\int_M \big(\Omega_{1}+\frac{1}{n-1}\big[(\Delta_{g} \tilde \phi_\infty )\Omega_{0}-\partial\partial_J \tilde \phi_\infty \big]\big)^{n}\wedge \bar \Omega_0^n}{\int_M \mathrm{e}^H \Omega^n_0\wedge \bar \Omega_0^n}\,.
	\end{equation*}	
	The complex version of flow \eqref{n-1 qpsh flow} was studied by Gill \cite{Gill2} as a parabolic approach to the complex Monge-Amp\`ere equation for $(n-1)$-plurisubharmonic functions, which originally arose from superstring theory in the works of Fu, Wang and Wu \cite{FWW10,FWW15}, and was then solved by Tosatti and Weinkove \cite{TW17,TW19} (see also \cite{CHZ,HZ}). As proven in \cite{GZ}, the solvability of \eqref{n-1 qpsh flow 1} leads to Calabi-Yau--type theorems for quaternionic balanced, quaternionic Gauduchon, and quaternionic strongly Gauduchon metrics. Therefore, convergence of flow \eqref{n-1 qpsh flow} results to be an interesting tool in the search of special metrics.
	\medskip	
	
	Going back to the discussion of the bounded case we observe that, unfortunately, under this assumption $\Gamma$-admissible functions might not be $\mathcal{C}$-subsolutions. Compared to Theorem \ref{unbounded theorem} the main result in the bounded case looks a little bit more artificial, as it requires some additional assumptions.
	\begin{theorem}\label{bounded theorem}
		Suppose $f$ is bounded on $\Gamma$. Let $(M,I,J,K,g,\Omega_0)$ be a compact flat hyperk\"ahler manifold. For any $\Gamma$-admissible initial datum $\varphi_{0}$, let $\varphi\in C^\infty(M\times[0,T),\R) $ be the maximal solution of flow \eqref{eq_main}. Assume further that
		\begin{enumerate}
			\itemsep0.2em
			\item[(i)] either it holds 
			\begin{equation}\label{case-1}
				\partial_{t}\underline{\varphi}\geq \sup_{M}\left(F(A[\varphi_{0}])-h\right)\,;
			\end{equation}
			\item[(ii)] or there exists a  non-increasing function $\Phi$ of class $C^{1}$ on $\mathbb{R}$ such that
			\begin{equation}\label{case-2}
				\begin{cases}
					\sup_M \left(\varphi(\cdot,t)-\underline{\varphi}(\cdot,t)-\Phi(t)\right)\geq 0\,,\\
					\sup_M \big(\varphi(\cdot,t)-\Phi(t)\big)\leq -C \inf_M \left(\varphi(\cdot,t)-\Phi(t)\right)+C\,
				\end{cases}
			\end{equation}
		\end{enumerate}
		for all $t\in (0,T)$ and a time-independent positive constant $C$. Then $T=\infty$, i.e. the solution $\phi$ exists for all times, and the normalization $\tilde{\varphi}$ converges smoothly to a function $\tilde{\varphi}_{\infty}\in C^\infty(M,\R)$ as $t\rightarrow\infty$, which solves \eqref{elliptic function} for some $b\in\mathbb{R}$.
	\end{theorem}
	
	Within the bounded case various equations can be included, for instance, parabolic quaternionic Hessian quotient equations, parabolic quaternionic mixed Hessian equations. We limit ourselves to prove the following general result.		
	\begin{theorem}\label{theorem 5}
		Suppose $f$ is bounded on $\Gamma$. Let $(M,I,J,K,g,\Omega_0)$ be a compact flat hyperk\"ahler manifold. If there exists a $\Gamma$-admissible function $\phi_0$ and a $\mathcal{C}$-subsolution of the equation
		\[
		F(A[\phi])=h
		\]
		in the sense of \cite{GZ}. Then there exists a smooth solution of the equation
		\[
		F(A[\phi])=h+b
		\]
		for some constant $b\in \R$.
	\end{theorem}
	
	\medskip
	The organization of the paper is the following. Our main results (Theorems \ref{unbounded theorem} and \ref{bounded theorem}) are proved via a fairly standard technique, that requires a priori estimates for the solution of our flow and its normalization. After a brief discussion of preliminaries in Section \ref{pre}, we prove estimates of order zero in Section \ref{C0}. Section \ref{Laplacian} is then devoted to finding a bound for the quaternionic Laplacian of the solution, in terms of the norm of its gradient. This type of bound is suitable to apply an interpolation argument that gives an estimate for the gradient; this is explained in Section \ref{Grad}. The aim of Section \ref{higher} is to apply an Evans-Krylov--type Theorem to deduce a $C^{2,\alpha}$ estimate, which is readily improved to higher order estimates via a standard bootstrapping argument. Long-time existence is also proved here. The convergence of flow \eqref{eq_main} is then showed in Section \ref{conv}. Finally, in the last Section, we employ these estimates to present a proof of Theorems \ref{unbounded theorem} and \ref{bounded theorem} as well as Theorems \ref{theorem Hessian}, \ref{theorem n-1 qpsh flow} and \ref{theorem 5}	which readily follow.
	\medskip

	\noindent{\bf Notation.} Let $Q\subseteq M\times [0,\infty)$ and fix $\alpha\in (0,1]$. We write
	\begin{itemize}
		\item $u\in C^{0,\alpha}(Q)$ if there exists a positive constant $C$ such that for $(x,s)$, $(y,t)\in Q$, the following holds
		\[ |u(x,s)-u(y,t)|\leq C(|x-y|^{\alpha}+|t-s|^{\frac{\alpha}{2}}). \]
		\item  $u\in C^{1,\alpha}(Q)$ if $u$ is $C^{\frac{\alpha+1}{2}}$ in time and $\nabla u$ is $C^{\alpha}$ in space.
		\item  $u\in C^{2,\alpha}(Q)$ if $\partial_tu$ is $\frac{\alpha}{2}$-H\"older in time and $\nabla^{2} u$ is $C^{\alpha}$ in space.
	\end{itemize}
	
	\medskip
	
	\noindent {\bf Acknowledgements.} The first author warmly thanks Luigi Vezzoni for continuous support and useful discussions.  The second author thanks  Prof. Xi Zhang for his constant encouragements. Both authors are grateful to the referee for a very careful reading of the paper, its useful comments and improvements; in particular, the referee suggested to use interpolation theory to obtain the gradient estimate in Section \ref{Grad}, which simplifies our previous proof.

	\section{Preliminaries}\label{pre}
	In this section we describe some basic useful preliminaries in hyperhermitian geometry and investigate a little bit the notion of parabolic $\mathcal{C}$-subsolution.
	\medskip
	
	A smooth manifold $M$ of real dimension $4n$ is called {\em hypercomplex} if it is equipped with three complex structures $I,J,K$ that behave like the standard quaternion units:
	\[
	IJ=-JI=K\,,
	\]
	we call the triple $(I,J,K)$ a {\em hypercomplex structure} on $M$.
	
	A hypercomplex structure is said to be {\em locally flat} if $M$ is locally isomorphic to the flat space $\H^n$ of $n$-tuples of quaternions. By this, we mean that, locally, $(I,J,K)$ is the pull-back of the standard hypercomplex structure on $\H^n$ given by the unit quaternions $\mathrm{i,j,k}$. By abuse of language if $(I,J,K)$ is locally flat we also say that the hypercomplex manifold $M$ is locally flat. This class of manifolds was originally studied by Sommese \cite{So}. As explained e.g. in \cite{GZ}, on a locally flat hypercomplex manifold $(M,I,J,K)$ we can introduce {\em quaternionic coordinates} $(q^1,\dots,q^n)$ around any point. This is a major difference with respect to complex and real manifold which always admit neighborhoods isomorphic to open subsets of $ \C^n $ and $ \R^n $ respectively.
	\medskip

	Let $(M,I,J,K,g)$ be a {\em hyperhermitian manifold}, which is a hypercomplex manifold endowed with a hyperhermitian metric $g$, i.e. is a Riemannian metric that is Hermitian with respect to each of $I,J,K$. To the hyperhermitian metric is naturally associated a $(2,0)$-form 
	$$
	\Omega_{0}=\omega_J+i\omega_K
	$$ 
	where $\omega_J$ and $\omega_K$ are the fundamental forms of $(g,J)$ and $(g,K)$ respectively. The form $\Omega_0$ will be said to be induced by the hyperhermitian metric $g$ and satisfies some nice properties, namely it is q-real and positive. A $(2p,0)$-form $\alpha$ on a hypercomplex manifold $(M,I,J,K)$ is called {\em q-real} if $J\bar{\alpha}=\alpha$, where the action of $J$ is given by $J\alpha=\alpha(J\cdot,\cdots,J\cdot)$; and it is called positive if it is q-real and $\alpha(Z_1,J\bar Z_1,\dots,Z_{2p},J\bar Z_{2p})>0$ for every vector fields $Z_1,\dots,Z_{2p}$ on $M$ of type $(1,0)$.
	
	A hyperhermitian manifold $(M,I,J,K,g)$ is called hyperk\"ahler if for any induced complex structure $L\in \H$ with $L^2=-1$, the corresponding Hermitian manifold $(M,L,g)$ is K\"ahler. It is well-known that $(M,I,J,K,g)$ is hyperk\"ahler if and only if $d\Omega_0=0$. The weaker assumption $\partial \Omega_0=0$ is the defining condition for the realm of HKT geometry. 
	\medskip
	
	Let $(M,I,J,K,g)$ be a locally flat hyperhermitian manifold and $\Omega$ a q-real $(2,0)$-form. Equation \eqref{eq_main} can be expressed in terms of the matrix
	\[
	A^r_s[\phi]=g^{\bar j r} \Omega_{\bar j s}^\phi=g^{\bar j r}\left( \Omega_{\bar j s}+ \phi_{\bar j s} \right)
	\]
	where $(g^{\bar j r})$ is the inverse matrix of $(g_{\bar j r})$ and $(\phi_{\bar j s})$ denotes the hyperhermitian matrix associated to $\partial \partial_J \phi$. Recall that a quaternionic matrix $H\in \H^{n,n}$ is called {\em hyperhermitian} if $^t\bar{H}=H$. With respect to quaternionic local coordinates $(q^1,\dots,q^n)$ it is well-known that
	\[
	\phi_{\bar r s}=\frac{1}{4} \partial_{\bar q^r} \partial_{q^s} \phi=:\mathrm{Hess}_\H \phi \,,
	\]
	where the so-called {\em Cauchy-Riemann-Fueter} operators $\partial_{\bar q^r}$ and $\partial_{q^s}$ act on smooth $ \H $-valued functions as follows 
	$$
	\partial_{	\bar q^r}u:=\sum_{i=0}^3e_{i}\, \partial_{x_i^r}u\,, \qquad \partial_{q^r}u:=\partial_{x_0^r}u\, e_0-\sum_{i=1}^3 \partial_{x_i^r}u\, e_i\,;
	$$
	here we are denoting the unit quaternions $1,\mathrm{i,j,k}$ with $e_0,e_1,e_2,e_3$ and we are taking the derivatives with respect to the real coordinates underlying the quaternionic ones, according to the relation $q^r=x^r_0+x^r_1\mathrm{i}+x^r_2\mathrm{j}+x^r_3\mathrm{k}$. The operators $\partial_{\bar q^r}$ and $\partial_{q^s}$ commute, but in general they do not satisfy the Leibniz or the chain rule, so care must be taken during computations.	
	Taking the real part of the trace of the {\em quaternionic Hessian} $\mathrm{Hess}_\H u$ with respect to the metric $g$ we have a second order linear elliptic operator called the {\em quaternionic Laplacian}
	\begin{equation*}
		\Delta_g u:=\Re\,\mathrm{tr}_g(\mathrm{Hess}_\H u )=\Re \left(g^{\bar jr} u_{\bar j r}\right).
	\end{equation*}
	More generally, when the manifold is not necessarily locally flat, the quaternionic Laplacian allows an intrinsic definition as
	\[
	\Delta_g u:=n\frac{\partial \partial_J u\wedge \Omega_0^{n-1}}{\Omega_0^n}\,,
	\]
	where $ \Omega_0$ is the $ (2,0) $-form induced by $g$. By \cite[Lemma 3]{BGV}, when the manifold is locally flat we recover the previous definition. Note that in quaternionic local coordinates $\Delta_g u $ is the sum of the eigenvalues of $\mathrm{Hess}_\H u$ with respect to $ g $.
	\medskip

	Now we briefly discuss the notion of $\mathcal{C}$-subsolution. Sz\'{e}kelyhidi introduced it in \cite{Sze} for elliptic equations. His definition is also shown to be a relaxation of that given by Guan \cite{Guan}. As for the parabolic case, Guan, Shi and Sui \cite{GSS} worked on Riemannian manifolds with the classical notion of a subsolution, while Phong and T\^o provided in \cite{Phong-To} the extension to the parabolic case of Sz\'{e}kelyhidi's definition. Of course, as we shall see in a moment with a characterization of $\mathcal{C}$-subsolutions, what happens in hyperhermitian geometry is entirely parallel to the Hermitian case. Thus, Definition \ref{Def} is the right extension of the one given in \cite{GZ} for the elliptic case. We shall refer to $\mathcal{C}$-subsolutions in the sense of \cite{GZ} as \emph{elliptic} ones.
	
	\begin{lemma}
		Let $\underline{\phi}\in C^{1,1}(M\times [0,+\infty))$ be such that $\|\underline{\phi}\|_{C^{1,1}}<+\infty$. Then $\underline{\phi}$ is a parabolic $\mathcal{C}$-subsolution if and only if there exists a uniform constant $\rho>0$ such that
		\[
		\lim_{s\to \infty}f(\lambda[\underline{\phi}(x,t)]+se_i)-\partial_t \underline \phi (x,t)>\rho+h(x)
		\]
		for each $i=1,\dots,n$, where $e_i$ is the $i$\textsuperscript{th} standard basis vector of $\R^n$. In particular when $\underline{\phi}$ is time-independent it is a $\mathcal C$-subsolution in the parabolic sense if and only if it is such in the elliptic sense.
	\end{lemma}
	\begin{proof}
		The proof can be reproduced almost verbatim from \cite[Lemma 8]{Phong-To}.
	\end{proof}
	
	This lemma in particular implies that when $f$ is unbounded over $\Gamma$, every $\Gamma$-admissible function is a parabolic $\mathcal{C}$-subsolution.
	\medskip
	
	We conclude this section by fixing some notations. Unless otherwise stated we shall always denote by $\phi$, $\tilde \phi$ and $\underline \phi$ the maximal solution to flow \eqref{eq_main} with initial datum $\phi_0$, its normalization as in \eqref{tilde varphi} and a parabolic $\mathcal{C}$-subsolution in the sense of Definition \ref{Def}, respectively. All these functions are assumed to be defined over $M\times[0,T)$, where $(M,I,J,K,g)$ is a compact locally flat hyperhermitian manifold and $T$ is the maximal time of existence of $\phi$.
	
	From here on, we will always denote with $ C $ a positive constant that only depends on background data (not on time!), including the initial datum $\phi_0$. Occasionally we might say that $C$ is {\em uniform}, to stress that it is time-independent. As it is customary, the constant $C$ may change value from line to line. 
	
	\section{$C^0$ estimates}\label{C0}	
	In this section we achieve estimates of order zero for the solution $\phi$ and its normalization $\tilde \phi $. We start by bounding their time derivatives, then, in order to treat the bounded case we need an additional inequality proved in Lemma \ref{lower bound unbounded case}. Such lemma follows as an application of the parabolic version of the Alexandroff-Bakelman-Pucci (ABP) inequality due to Tso \cite[Proposition 2.1]{Tso} by adapting the argument of Phong-T\^o \cite[Lemma 1]{Phong-To}.
	\medskip
	
	\subsection{Bounds on $\partial_{t}\varphi$ and $\partial_t \tilde{\phi}$.}	
	
	\begin{lemma}\label{Bounds on partial-t}
		We have
		\begin{equation}\label{up and low of phit}
			\inf_{M}\big(F(A[\varphi_{0}])-h\big)\leq \partial_{t}\varphi\leq  \sup_{M}\big(F(A[\varphi_{0}])-h\big)
		\end{equation}
		and
		\begin{equation}\label{up and low of tilde phit}
			\left \lvert \partial_{t} \tilde \varphi \right\rvert \leq C\,,
		\end{equation}
		for a uniform constant $C>0$ depending only on $h$ and the initial datum $\phi_0$. 
	\end{lemma}	
	
	\begin{proof}
		Differentiating the flow \eqref{eq_main} along $\partial_t$ we see that $\partial_{t}\varphi$ satisfies the following heat type equation
		\begin{equation}\label{heat equation}
			\partial_{t}\big(\partial_{t}\varphi\big)=\frac{1}{4}\Re\, \big(F^{rs }\partial_{\bar{q}^r}\partial_{q^s}(\partial_t \varphi)\big)\,,
		\end{equation}
		where $ F^{rs}:=\frac{\partial F}{\partial A_{rs}}.$ By the parabolic maximum principle for \eqref{heat equation}, we know that $\partial_{t}\varphi$  hits its extremum  at $t=0$. Thus,
		\begin{equation*}
			\inf_{M\times\{0\}}\partial_{t}\varphi\leq \partial_{t}\varphi\leq \sup_{M\times\{0\}}\partial_{t}\varphi\,, \qquad \partial_{t}\varphi(\cdot,0)=F(A[\varphi_{0}])-h
		\end{equation*}
		and we then obtain \eqref{up and low of phit}. The bound \eqref{up and low of tilde phit} on $|\partial_t \tilde \phi|$ follows immediately.
	\end{proof}
	We remark that a direct consequence of the previous lemma is the following short-time estimate:
	\begin{equation}\label{local c0}
		|\varphi|\leq C\delta \,, \qquad \text{on} \, M\times [0,\delta]\,.
	\end{equation}
	
	\medskip	
	
	\subsection{Intermediate bounds.}	
	\begin{lemma}\label{wsubh}
		Let $\psi$ be a smooth function on $M\times [0,T)$ satisfying
		\begin{equation}\label{weak subharmonic}
			\Delta_{g}\psi\geq c_{0}   
		\end{equation}
		for a uniform constant $c_{0}\in \R$, then there exist $ p,C>0 $, depending only on the background data, such that
		\[
		\|\psi-\sup_{M}\psi\|_{L^p(M)}\leq C\,.
		\]
	\end{lemma}
	\begin{proof} 
		The proof can be found in \cite[Lemma 5]{GZ}, for convenience of the reader we shall also sketch it here. Take an open cover of $ M $ made of coordinate balls $ B_{2r_i}(x_i) $ such that $ \{ B_i=B_{r_i}(x_i) \} $ still covers $ M $. Set $\Psi=\psi-\sup_{M}\psi$ for simplicity. By the elliptic inequality \eqref{weak subharmonic} we can apply the weak Harnack inequality \cite[Theorem 9.22]{GT} deducing
		\begin{equation}\label{weak Harnack}
			\|\Psi\|_{L^p(B_i)}\leq C\left( \inf_{B_i}(-\Psi)+1 \right)
		\end{equation}	
		where $ p,C>0 $ depend only on the cover and the background metric. Since $ \Psi $ is non-positive there is at least one coordinate ball $B_j$ such that $ \inf_{B_j}(-\Psi)=0 $, and thus $ \|\Psi\|_{L^p(B_j)}\leq C $. The bound on $ \|\Psi\|_{L^p(B_j)} $ also gives a bound for $ \inf_{B_i}(-\Psi) $ on all coordinate balls intersecting $B_j$. Using again \eqref{weak Harnack} and repeating the argument we obtain an upper bound on each ball of the cover.
	\end{proof}
	
	\begin{lemma}\label{lower bound unbounded case}
		If there exists a non-increasing function $\Phi\in C^1([0,T),\R)$ satisfying
		\begin{equation*}\label{normalized condition}
			\sup_M \left( \phi(\cdot,t)-\underline{\phi}(\cdot,t)-\Phi(t)\right)\geq 0\,,  
		\end{equation*}
		then there exists a constant $C>0$, depending only on $\Omega, g, \underline{\phi}, \|\phi_0\|_{C^0}$ such that
		\begin{equation*}\label{estimate of unbounded}
			\phi(x,t)-\underline \phi (x,t) - \Phi(t) \geq - C \qquad \text{for all }(x,t)\in M\times [0,T)\,.
		\end{equation*}
	\end{lemma}
	
	\begin{proof}
		First, observe that the requirement $\Phi' \leq 0$ implies that $\underline \phi +\Phi$ is still a parabolic $\mathcal{C}$-subsolution of \eqref{eq_main}, therefore, as long as the involved constants do not depend on the time derivative of $\underline{\phi}$, we may assume $\Phi\equiv 0$.
		
		Choose $\delta\in (0,1) $ and $R>0$ such that \eqref{cone condition} holds for the subsolution $\underline \phi$. By \eqref{local c0}, it  suffices to estimate $v=\varphi-\underline{\varphi}$ on $M\times[\delta,T).$ Fix an arbitrary $T'< T$ and assume $v $ achieves its minimum $S$ at a point $(x_{0},t_{0})\in M\times [\delta,T']$, i.e.,
		\begin{equation*}
			S=v(x_{0},t_{0})=\min_{M\times [\delta,T']}v\,. 
		\end{equation*}
		Now we are reduced to prove that if $\sup_{M}v\geq 0$ for all $t\in [\delta,T']$, then $S$ is bounded from below by a constant depending only on $\Omega, g, \underline{\phi}, \|\phi_0\|$ and independent of $T'$.	
		
		Consider quaternionic local coordinates $ (q^1,\dots,q^{n}) $ centered at the point $x_{0}$. We may identify such coordinate neighborhood with the open ball of unit radius $ B_1=B_1(0)\subseteq \H^{n} $ centered at the origin. Let 
		\[ w(x,t)=v(x,t)+\frac{\delta^{2}}{4}|x|^2 +(t-t_0)^{2},\]
		be a function defined on $\mathcal{B}= B_1\times [t_{0}-\frac{\delta}{2},t_{0}+\frac{\delta}{2}] $. Observe that $ \inf_{\mathcal{B}} w=w(0,t_{0})=v(0,t_{0})=S $ and $ \inf_{\partial \mathcal{B}}w\geq w(0,t_{0})+\frac{\delta^{2}}{4} $. These conditions allow us to apply the parabolic ABP method of Tso \cite[Proposition 2.1]{Tso} to obtain
		\begin{equation}
			\label{eq_in_C0_3}
			C_0\delta^{8n+2}\leq \int_P |\partial_{t}w| \det(D^2w)\,,
		\end{equation}
		where $ C_0>0 $ is a dimensional constant,
		\[
		P=\left\{(x,t)\in \mathcal{B}\ \Bigg| \	\begin{aligned}& w(x,t)\leq S+\frac{\delta^{2}}{4},\quad |Dw(x,t)|<\frac{\delta^{2}}{8},\\ 
			w(y,s) & \geq w(x,t)+Dw(x,t)\cdot (y-x), \, \forall y\in B_{1},\quad s\leq t \end{aligned}\right\}
		\]
		is the parabolic contact set of $w$	on $\mathcal{B}$ and $ Dw $, $ D^2w $ are the gradient and the (real) Hessian of $ w $ on $M$ with respect to the variable $x$. 
		\medskip
		~\\
		\textbf{Claim:} both $|\partial_{t}w|$ and $\det(D^{2}w)$ are bounded on $P$.
		\medskip			
		
		Let
		$\tau=-\partial_{t}\varphi+\partial_{t}\underline{\varphi}=-\partial_{t}v$ and $\mu=\lambda(A[\varphi])-\lambda(A[\underline{\varphi}])$. Observe that $D^{2}w\geq 0$ and $\partial_{t}w\leq 0$ on $P$. Thus, 
		\[
		\tau=-\partial_{t}w+2(t-t_{0})\geq -\delta\,, \qquad \mu+\delta \textbf{1}\in \Gamma_{n}\,.
		\]
		Now by Definition \ref{Def} we conclude that $|\tau|+|\mu|\leq R$, then $|\partial_{t}w|\leq R$ and $\mathrm{Hess}_\H w$ is a bounded matrix. But then we are done as we have
		\[
		\det(D^2w) \leq 2^{4n}\det(\mathrm{Hess}_\H w)^4\,\quad \text{on} \ P,
		\]
		which follows from a computation in \cite{Blocki} and \cite[Lemma 2]{Sro}, (see also the proof of \cite[Proposition 2.1]{Alesker-Shelukhin (2013)}). Here, on the right-hand side, ``$ \det $'' denotes the Moore determinant, introduced in \cite{Moore} (see also e.g. \cite{Alesker (2003),Aslaksen,SThesis}). This confirms  the claim.
		\medskip	
		
		With this claim at hand, by \eqref{eq_in_C0_3} we have
		\begin{equation}\label{abp-1}
			C_0\delta^{8n+2}\leq C \mathrm{Vol}(P)\,.
		\end{equation}
		From \eqref{Gamma in Gamma1} we readily obtain $ \Re\,\mathrm{tr}_g(\Omega_\phi)>0 $, where $ \Omega_\phi =\Omega+\partial \partial_J\phi $, which in turn yields a uniform lower bound for the quaternionic Laplacian of $ \phi $:
		\begin{equation*}
			\Delta_g \phi=\Re\, \mathrm{tr}_g(\Omega_\phi)-\Re\,\mathrm{tr}_g(\Omega) \geq -C\,.
		\end{equation*}
		This also gives a uniform lower bound for $\Delta_{g}v$. Then by 
		Lemma \ref{wsubh}, we see that
		\begin{equation}\label{Lp}
			\|v-\sup_{M}v\|_{L^p(M)}\leq C.
		\end{equation}
		The definition of $ P $ and our assumption that $\sup_{M}v\geq 0$ on $[0,T)$ yields
		\[
		v-\sup_{M}v\leq v\leq w<S+\frac{\delta^{2}}{4} \quad \text{on} \ P,
		\]
		We may further assume $S+\frac{\delta^{2}}{4}<0$, otherwise we are done.
		As a consequence for any $ p>0 $
		\[ 
		\left \lvert S+\frac{\delta^{2}}{4}\right \rvert^p\mathrm{Vol}(P)\leq \int_{P}|v-\sup_{M}v|^{p}dx dt\leq \int_{[t_{0}-\frac{\delta}{2},t_{0}+\frac{\delta}{2}]}\|v-\sup_{M}v\|^{p}_{L^p(M)}dt\leq C\delta\,,
		\]
		where we have used \eqref{Lp}. This, together with \eqref{abp-1}, gives the uniform lower bound of $S$ we were after.
	\end{proof}
	
	\medskip
	
	\subsection{Bounds on $\varphi$ and $\tilde{\phi}$.}\hfill\\
	\vskip-0.4em
	
	As it often happens for solutions to flows, we only manage to control the oscillation and not the full $C^0$ norm. On the other hand, once the oscillation is under control, we immediately achieve the $C^0$ estimate for the normalization of the solution.
	\begin{proposition}\label{c0 estimate}
		Let $f$ be either bounded or unbounded. In case $f$ is bounded on $\Gamma$ assume that it satisfies either one of the two conditions expressed in Theorem \ref{bounded theorem}. Then there exists a uniform constant $C>0$, depending only on the background data such that
		\begin{equation}\label{bound osc}
			\osc_M\phi(\cdot,t):=\sup_M \phi(\cdot,t)-\inf_M \phi(\cdot,t) \leq C\,,
		\end{equation}
		and
		\begin{equation}\label{C0 bound tilde phi}
			\|\tilde{\phi}\|_{C^0}\leq C\,.
		\end{equation}
	\end{proposition}
	\begin{proof}
		First, we observe that \eqref{C0 bound tilde phi} follows from \eqref{bound osc}. Indeed, by the normalization of $\tilde \phi $, for any $(x,t)\in M\times [0,T)$ we can find $y(x)\in M$ such that $\tilde{\phi}(y(x),t)=0$, therefore
		\[
		\|\tilde{\phi}\|_{C^0}=\sup_{(x,t)\in M} \big |\tilde{\phi}(x,t)-\tilde{\phi}(y(x),t) \big |=\sup_{(x,t)\in M} \big|\phi(x,t)-\phi(y(x),t)\big|\leq \osc_M \phi(\cdot,t)\,.
		\]
		
		We will prove \eqref{bound osc} by rewriting the flow \eqref{eq_main} as
		\begin{equation}\label{proof osc}
			F(A[\phi])=h+\partial_t \phi\,,
		\end{equation}
		and interpreting it for every fixed time as an elliptic equation with datum $h+\partial_t \phi$. We split the argument into two cases according as $f$ is bounded or unbounded.
		\begin{itemize}
			\item Case 1. $f$ is unbounded on $\Gamma$. In this case any $\Gamma$-admissible function is a parabolic $\mathcal{C}$-subsolution, therefore we can take the initial datum $\phi_0$ as such. Since $\phi_0$ is time-independent, it can be regarded as an elliptic $\mathcal{C}$-subsolution. Furthermore, by Lemma \ref{Bounds on partial-t} we know that the right-hand side of \eqref{proof osc} is uniformly bounded, therefore we may apply \cite[Proposition 6]{GZ} to obtain \eqref{bound osc}.
			\item Case 2. $f$ is bounded on $\Gamma$. We consider two subcases. Assume that condition (i) of Theorem \ref{unbounded theorem} holds, then \eqref{case-1} and Lemma \ref{Bounds on partial-t} imply that $\partial_t \underline \phi \geq \partial_t \phi$, this entails that $\underline \phi$ is a $\mathcal{C}$-subsolution of \eqref{proof osc} in the elliptic sense. Again \eqref{bound osc} follows from \cite[Proposition 6]{GZ}. If, instead, condition (ii) of Theorem \ref{unbounded theorem} is satisfied, then there exists $\Phi\in C^\infty([0,T),\R)$ with $\Phi'\leq 0$ satisfying \eqref{case-2} and we can readily apply Lemma \ref{lower bound unbounded case} to conclude. \qedhere
		\end{itemize}
	\end{proof}		
	
	\smallskip

	\section{Quaternionic Laplacian estimate}\label{Laplacian}
	Here we adopt the technique of \cite{Chou-Wang,Hou-Ma-Wu} which allows to find a Laplacian bound in terms of the norm of the gradient.
	\medskip
	
	Before we tackle the proof, we recall the following preliminary lemma given in Phong-T\^o \cite[Lemma 3]{Phong-To}, which was inspired by the elliptic version of \cite[Proposition 6]{Sze}. 
	We will use the following derivatives of $F$
	\[
	F^{rs}:=\frac{\partial F}{\partial A_{rs}}\,, \qquad F^{rs,lt}:=\frac{\partial^2 F}{\partial A_{rs}\partial A_{lt}}\,.
	\]
	
	\begin{lemma}\label{pre C2}
		Let $\delta,R$ be uniform constants such that on $M\times[0,T)$, if $(\mu,\tau)\in \R^n\times \R $ satisfy \eqref{cone condition}, then $|\mu|+|\tau|<R$. There exists a uniform constant $ \kappa>0 $ depending on $\delta $ and $ R $  such that if $  |\lambda(A[\varphi])-\lambda(A[\underline{\varphi}])|>R $, we have 
		\begin{align*}
			&\mbox{either }&& \Re\, F^{rs}(A[\varphi])\left( A_{rs}[\underline{\varphi}]-A_{rs}[\varphi]\right)-(\partial_{t}\underline{\varphi}-\partial_{t}\varphi)>\kappa \sum_{r=1}^n F^{rr}(A[\varphi])\,,\\
			&\mbox{or }&& F^{ss}(A[\varphi])>\kappa \sum_{r=1}^n F^{rr}(A[\varphi])\,, \qquad \text{for all } s=1,\dots,n\,.
		\end{align*}
	\end{lemma}
	
	\begin{proof}
		Since the quaternionic analogue of the Schur-Horn theorem holds, see e.g. \cite[Lemma 8]{GZ}, the proof of the lemma can be adapted from \cite[Lemma 3]{Phong-To}.
	\end{proof}
	
	\begin{proposition} \label{prop_C2-bound}
		Suppose $ (M,I,J,K,g) $ is a compact flat hyperk\"ahler manifold. Then there is a constant $ C>0 $, depending only on $ (M,I,J,K) $, $ \|g\|_{C^2} $, $ \|h\|_{C^2} $, $ \|\Omega \|_{C^2} $, $\|\underline{\phi}\|_{C^{1,1}} $, $\|\partial_t \phi \|_{C^0}$ and $ \|\tilde \phi \|_{C^0} $, such that
		\[
		\|\Delta_g \phi \|_{C^0}\leq C\left( \|\nabla \phi \|_{C^0}+1 \right)\,.
		\]
	\end{proposition}
	\begin{proof}
		By \eqref{Gamma in Gamma1} we already know that the quaternionic Laplacian is uniformly bounded from below, therefore it is enough to obtain a bound of the form
		\[
		\frac{\lambda_1}{\|\nabla \phi\|_{C^0}+1}\leq C\,,
		\]
		where $\lambda_1$ is the largest eigenvalue of $A[\phi]$. Let $T'<T$, all computations will be performed in quaternionic local coordinates around some fixed point $p_0=(x_0,t_0)\in M\times [0,T']$ which we will specify in a moment. As pointed out by Sz\'{e}kelyhidi \cite{Sze} in order for $ \lambda_1\colon M\to \R $ to define a smooth function at $ p_0 $ we need the eigenvalues to be distinct; to be sure of that, we perturb the matrix $ A $ as follows. Using the assumption that $g$ is a flat hyperk\"ahler metric we may take quaternionic coordinates such that $(g_{\bar r s})$ is the identity in the whole neighborhood of $p_0$  and $ (\Omega^{\phi}_{\bar rs}) $ is diagonal at $p_0$. In particular $ A[\phi] $ is diagonal with ordered eigenvalues $\lambda_1\geq \lambda_2\geq \dots \geq \lambda_n $. Let $ D $ be a constant diagonal matrix with entries satisfying $0=D_{11}<D_{22}<\dots<D_{nn}$. The matrix $ \tilde A=A[\phi]-D $ has distinct eigenvalues $\tilde \lambda_r$ by construction, and its largest eigenvalue $\tilde \lambda_1$ coincides with $\lambda_1$ at $p_0$.
		
		Choose $p_0\in M \times [0,T']$ to be a maximum point of the function
		\[
		\hat G=2\sqrt{\lambda_1}+\alpha(|\nabla \phi|^2)+\beta(\tilde v)
		\]
		where 
		\begin{align*}
			\alpha(s)&=-\frac{1}{2}\log \left( 1- \frac{s}{2N} \right)\,,& N&= \|\nabla \phi\|^2_{C^0}+1\,,\\
			\beta(s)&=-2Ss+\frac{1}{2}s^2\,,& S&>\|\tilde v\|_{C^0}\,, \text{ large constant to be chosen later}\,,
		\end{align*}
		and $\tilde v$ is the normalization of $v=\phi- \underline \phi$. As said, to avoid smoothness issues we shall not work with $\lambda_1$. Therefore, in a small neighborhood of $p_0$, instead of working with $\hat G$ we consider the function
		\[
		G=2\sqrt{\tilde \lambda_1}+\alpha(|\nabla \phi|^2)+\beta(\tilde v)\,.
		\]
		It will be useful to observe that
		\begin{align}\label{eq_alpha}
			\frac{1}{4N}&<\alpha'(|\nabla \phi|^2)<\frac{1}{2N}\,, & \alpha''&=2(\alpha')^2\,,\\
			\label{eq_beta}
			S&\leq -\beta'(\tilde v)\leq 3S\,, & \beta''&=1\,.
		\end{align}
		We also remark that, as in \cite{Phong-To}, at the point $p_0$ there exists a constant $ \tau>0 $ depending only on $ \|h\|_{C^0} $ and $\|\partial_t \phi\|_{C^0}$ such that
		\begin{equation*}
			\mathcal{F}:=\sum_{a=1}^n F^{aa}(A[\phi])=\sum_{a=1}^n f_{a}(\lambda(A[\phi]))>\tau\,.
		\end{equation*} \
		Indeed, for any $\sigma\in (\sup_{\partial\Gamma}f,\sup_{\Gamma}f)$ and by \cite{Sze}, there exists a constant $\tau'>0$ depending only on $\sigma$ such that $\sum_{a} f_{a}(\lambda)>\tau'$ for any $\lambda\in \partial\Gamma^{\sigma}$. Now for $f(\lambda(A[\phi]))=\partial_{t}\varphi+h$, $\sigma$ lies a compact set bounded by  $\|h\|_{C^0}+\|\partial_t \phi\|_{C^0}$ and whence $\mathcal{F}$ is bounded below by some $\tau$. This will be useful to absorb some constants during our computations.	
		
		The linearized operator $ L $ is defined by
		\[
		L(u)=4\sum_{a,b=1}^n F^{ab}g^{\bar c a}u_{\bar c b}-4\partial_{t}u\,,
		\]
		where $ u_{\bar a b}=\frac{1}{4}\partial_{\bar q^a}\partial_{q^b}u $. In particular, at $p_0$ the linearized operator has the simpler expression $L(u)=4(F^{aa}u_{\bar a a}-\partial_{t}u)$. We emphasize here that the terms $F^{aa}$ are real, a fact that we shall use in all the computations to come.
		
		At the maximum point $ p_0 $ we have $ L(G)\leq 0 $ i.e.
		\begin{equation}\label{eq_in_C2}
			0\geq L\Big(2\sqrt{\tilde \lambda_1}\Big)+L\Big(\alpha(|\nabla \phi|^2)\Big)+L\Big(\beta(\tilde v) \Big)\,.
		\end{equation}
		
		\medskip
		
		\subsection{Bound for $ L(2\sqrt{\tilde \lambda_1}) $.}\hfill\\
		\vskip-0.4em
		
		We claim that 
		\begin{equation}\label{eq bound L}
			L\Big(2\sqrt{\tilde \lambda_1}\Big)\geq -\frac{F^{aa}|\Omega^\phi_{\bar 1 1,a}|^2}{2\lambda_1\sqrt{\lambda_1}}-\frac{C\mathcal{F}}{\sqrt{\lambda_1}}\,,
		\end{equation}
		where $ \Omega^\phi_{\bar 1 1,a}=\partial_{q^a}\Omega^\phi_{\bar 1 1} $ and $ C>0 $ is a positive uniform constant.
		
		We clearly have 
		\begin{equation}\label{eq_pre_C2_2}
			\begin{split}
				L\Big(2\sqrt{\tilde \lambda_1}\Big)&=8F^{aa}\Big(\sqrt{\tilde \lambda_1}\Big)_{\bar a a}-8\partial_{t}\Big(\sqrt{\tilde \lambda_1}\Big)\\&=2F^{aa}\sum_{p=0}^3\Big(\sqrt{\tilde \lambda_1}\Big)_{x^a_px^a_p}-8\partial_{t}\Big(\sqrt{\tilde \lambda_1}\Big)\\
				&=\frac{1}{\sqrt{\lambda_1}}\Big(F^{aa}\sum_{p=1}^{3} \tilde \lambda_{1,x^a_px^a_p}-4\partial_{t}\tilde{\lambda}_{1}\Big)-F^{aa}
				\sum_{p=0}^3\frac{|\lambda_{1,x^a_p}|^{2}}{2\lambda_1\sqrt{\lambda_1}}  \,,
			\end{split}
		\end{equation}
		where the subscripts $ x^a_p $ denote the real derivative with respect to the corresponding real coordinates underlying the chosen quaternionic local coordinates. Using the formulas for the derivatives of the eigenvalues (see \cite{GZ}) and the fact that $ D $ is a constant matrix we obtain at $ p_0 $
		\begin{align*}
			\tilde
			\lambda_{1,x^a_p}&=\tilde
			\lambda_1^{rs}\tilde A_{rs,x^a_p}=\Omega^\phi_{\bar 1 1, x^a_p}\\
			\tilde
			\lambda_{1,x^a_px^a_p}&=\tilde \lambda_{1}^{rs,lt} \tilde A_{rs,x^a_p}\tilde A_{lt,x^a_p}+\tilde
			\lambda_{1}^{rs}\tilde A_{rs,x^a_px^a_p}=2 \sum_{r>1}\frac{|\Omega^\phi_{\bar r1,x^a_p}|^2}{\lambda_1-\tilde \lambda_r}+\Omega^\phi_{\bar 1 1, x^a_px^a_p}\,.
		\end{align*}
		Observe that
		\[
		\begin{split}
			\sum_{p=0}^3\Omega^\phi_{\bar 1 1, x^a_px^a_p}&=\sum_{p=0}^3\left(\Omega_{\bar 1 1, x^a_px^a_p}+\phi_{\bar 1 1 x^a_px^a_p}\right)=4\Omega_{\bar 1 1, \bar a a}+4\phi_{\bar a a \bar 1 1}\\
			&=4\Omega_{\bar 1 1, \bar a a}-4\Omega_{\bar a a, \bar 1 1}+\sum_{p=0}^3\Omega^\phi_{\bar a a, x^1_p x^1_p}    
		\end{split}
		\]
		which implies
		\[
		F^{aa}\tilde \lambda_{1,x^a_px^a_p}\geq F^{aa}\sum_{p=0}^3\Omega^\phi_{\bar a a, x^1_p x^1_p}-C\mathcal{F}\,.
		\]
		Differentiating the equation $\partial_{t}\varphi= F(A[\phi])-h $ twice with respect to $ x^1_p $ gives, at $ p_0 $,
		\begin{equation*}
			F^{rs,tl}\Omega^\phi_{\bar r s,x^1_p}\Omega^\phi_{\bar t l,x^1_p}+F^{aa} \Omega^\phi_{\bar a a,x^1_px^1_p}=h_{x^1_px^1_p}+\partial_{t}(\varphi_{x^1_px^1_p})\,.
		\end{equation*}
		by this and the concavity of $ F $
		\begin{equation}\label{eq bound L1}
			F^{aa}\sum_{p=0}^3\tilde \lambda_{1,x^a_px^a_p}-4\partial_{t}\tilde{\lambda}_{1}\geq \sum_{p=0}^3 \left(F^{aa}\Omega^\phi_{\bar a a, x^1_px^1_p}-\partial_{t}(\varphi_{x^1_px^1_p})\right)-C\mathcal{F}
			\geq -C\mathcal{F}\,.    
		\end{equation}
		Substituting \eqref{eq bound L1} into \eqref{eq_pre_C2_2} we obtained the claimed inequality \eqref{eq bound L}.
		
		\medskip
		
		\subsection{Bound for $ L\left(\alpha(|\nabla \phi|^2)\right) $.}\hfill\\
		\vskip-0.4em
		
		First of all, since $|\nabla \phi |^2=\sum_r \phi_{\bar r} \phi_r$ is real, we may compute
		\begin{equation}\label{eq bound Lalpha}
			\begin{split}
				L\left(\alpha(|\nabla \phi|^2)\right)=&\,\alpha''F^{aa}\sum_{p=0}^3\left(\sum_{r=1}^n(\phi_{\bar r x^a_p}\phi_{r}+\phi_{\bar r}\phi_{rx^a_p})\right)^2\\
				&+\alpha'F^{aa}\sum_{p=0}^3\sum_{r=1}^n\left(\phi_{\bar rx^a_px^a_p}\phi_{r}+2|\phi_{r x^a_p}|^2+\phi_{\bar r}\phi_{rx^a_px^a_p}\right)\\
				&-\alpha'\sum_{r=1}^n\Big(\partial_{t}(\phi_{\bar r})\phi_{r}+\phi_{\bar r}\partial_{t}(\phi_{r})\Big)\,.
			\end{split}
		\end{equation}
		Differentiating the equation $ \partial_{t}\varphi=F(A[\varphi])-h $ yields
		\[
		\partial_{t}(\varphi_{x^r_p})=F^{aa}\Omega^\phi_{\bar aa,x^r_p}-h_{x^r_p}\,, \qquad \text{at }p_0\,.
		\]
		Together with Cauchy-Schwarz inequality and \eqref{eq_alpha} this yields
		\[
		\begin{split}
			&\, \alpha'F^{aa}\sum_{r=1}^n(\phi_{\bar r \bar a a}\phi_r+\phi_{\bar r}\phi_{r\bar a a})-\alpha'\sum_{r=1}^n\Big(\partial_{t}(\phi_{\bar r})\phi_{r}+\phi_{\bar r}\partial_{t}(\phi_{r})\Big)\\
			&=\alpha'\sum_{r=1}^n\left((h_{\bar r}-F^{aa}\Omega_{\bar aa, \bar r})\phi_r+\phi_{\bar r}(h_{r}-F^{aa}\Omega_{\bar aa, r})\right)\\
			&\geq -\frac{C}{N}(N^{1/2}+N^{1/2}\mathcal{F})\geq -C\mathcal{F}\,,
		\end{split}
		\]
		Moreover, we have
		\[
		\begin{split}
			2\alpha'F^{aa}\sum_{r=1}^n\sum_{p=0}^3|\phi_{r x^a_p}|^2
			&
			\geq \frac{1}{2N}F^{aa}\sum_{p=0}^3\phi_{x^a_p x^a_p}^2=\frac{8}{N}F^{aa}\phi_{\bar a a}^2
			\\&=\frac{8}{N}F^{aa}(\lambda_a-\Omega_{\bar aa})^2\geq \frac{4}{N}F^{aa}\lambda_a^2-C\mathcal{F}\,.
		\end{split}
		\]
		
		Combining the last two inequalities with \eqref{eq bound Lalpha} we get
		\begin{equation}\label{bound alpha}
			L\left(\alpha(|\nabla \phi|^2)\right)\geq \alpha''F^{aa}\sum_{p=0}^3\left(\sum_{r=1}^n(\phi_{\bar r x^a_p}\phi_{r}+\phi_{\bar r}\phi_{rx^a_p})\right)^2+\frac{4}{N}F^{aa}\lambda_a^2-C\mathcal{F}\,.
		\end{equation}

		\subsection{Conclusion of the proof.}\hfill\\
		\vskip-0.4em
		
		In view of \eqref{eq bound L} and \eqref{bound alpha},	the main inequality \eqref{eq_in_C2} becomes
		\begin{equation}\label{eq_in_C2_2}
			\begin{split}
				0\geq \,\,&\alpha''F^{aa}\sum_{p=0}^3\left(2\sum_{r=1}^n\Re(\phi_{\bar r x^a_p}\phi_{r})\right)^2-\frac{F^{aa}|\Omega^\phi_{\bar 1 1,a}|^2}{2\lambda_1\sqrt{\lambda_1}}\\
				&+\frac{4F^{aa}\lambda_a^2}{N}+L\left(\beta(\tilde v) \right)-C\mathcal{F}
			\end{split}
		\end{equation}
		
		Since $p_0$ is a maximum point for $G$ we have
		\[
		0=G_{x^a_p}=\frac{\Omega^\phi_{\bar 1 1,x^a_p}}{\sqrt{\lambda_1}}+2\alpha'\sum_{r=1}^n\Re(\phi_{\bar r x^a_p}\phi_{r})+\beta'\tilde v_{x^a_p}
		\]
		and therefore, by \eqref{eq_alpha}
		\begin{equation}\label{eq_in_C2_3}
			\begin{split}
				\alpha'' F^{aa}\left(2\sum_{r=1}^n\Re(\phi_{\bar r x^a_p}\phi_{r})\right)^2&=2 F^{aa}\left( \frac{\Omega^\phi_{\bar 1 1,x^a_p}}{\sqrt{\lambda_1}}+\beta'\tilde v_{x^a_p}\right)^2\\
				&\geq 2\epsilon \frac{F^{aa}(\Omega^\phi_{\bar 1 1,x^a_p})^2}{\lambda_1}-\frac{2\epsilon }{1-\epsilon}(\beta')^2F^{aa}\tilde v_{x^a_p}^2\,,
			\end{split}
		\end{equation}
		where we used the inequality $ (a+b)^2\geq \epsilon a^2-\frac{\epsilon}{1-\epsilon}b^2\,, $ which holds for $ \epsilon\in (0,1) $. Assuming without loss of generality that $ \sqrt{\lambda_1}>\frac{1}{4\epsilon} $ we get 
		\begin{equation}\label{eq_in_C2_4}
			\left(4\epsilon\sqrt{\lambda_1}-1\right)\frac{F^{aa}|\Omega^\phi_{\bar 1 1,a}|^2}{2\lambda_1\sqrt{\lambda_1}} \geq 0\,.
		\end{equation}
		Putting together \eqref{eq_in_C2_3}, \eqref{eq_in_C2_4} and the calculation
		\[
		L\left( \beta(\tilde{v})\right)=\beta'' F^{aa}|\tilde{v}_{a}|^2+4\beta'F^{aa}\tilde{v}_{\bar a a}-4\beta'\partial_t \tilde{v}
		\]
		\eqref{eq_in_C2_2} simplifies to
		\[
		0\geq \frac{4F^{aa}\lambda_a^2}{N}+\left(\beta''-\frac{2\epsilon}{1-\epsilon}(\beta')^2\right) F^{aa}|\tilde{v}_{a}|^2+4\beta'\left(F^{aa}\tilde{v}_{\bar a a}-\partial_t \tilde{v}\right)-C\mathcal{F}\,.
		\]
		If we choose $ \epsilon=1/(18S^2+1)<1 $, then \eqref{eq_beta} yields
		\[
		\beta''-\frac{2\epsilon}{1-\epsilon}(\beta')^2\geq 0\,,
		\]
		therefore we finally arrive at
		\begin{equation}\label{eq_in_C2_5}
			0\geq \frac{4F^{aa}\lambda_a^2}{N}+4\beta'\left(F^{aa}\tilde{v}_{\bar a a}-\partial_t \tilde{v}\right)-C\mathcal{F}\,.
		\end{equation}
		Supposing $ \lambda_1>R $ we have $ |\lambda(A[\phi])|>R $ and we can then apply Lemma \ref{pre C2} according to which there exists $ \kappa>0 $ such that one of the following two cases occur:		
		\begin{itemize}
			\item
			Case 1: 
			$$
			\Re\, F^{rs}(A[\varphi])\left( A_{rs}[\underline{\varphi}]-A_{rs}[\varphi]\right)-(\partial_{t}\underline{\varphi}-\partial_{t}\varphi)>\kappa \sum_{r=1}^n F^{rr}(A[\phi])\,,
			$$
			i.e. $ -F^{aa} v_{\bar a a}+\partial_t v >\kappa \mathcal{F} $ at $ p_0 $, where we recall that $v=\phi-\underline \phi$. This immediately gives
			\[
			F^{aa} \tilde v_{\bar a a}-\partial_t \tilde v <-C_1 \mathcal{F}
			\]
			where $C_1$ depends on $\|\partial_t v\|_{C^0}$. Choosing $ S $ so large as to have $ \beta'(F^{aa}\tilde{v}_{\bar a a}-\partial_t \tilde{v})\geq C\mathcal{F} $ we deduce from \eqref{eq_in_C2_5} the inequality $ 0\geq \frac{4}{N}F^{aa}\lambda_a^2 $ which is a contradiction, hence this case cannot occur.
			
			\vspace{0.1cm}
			\item Case 2:
			\begin{equation*}
				F^{ss}(A[\phi])>\kappa \sum_{r=1}^n F^{rr}(A[\phi])\,, \qquad \text{for all } s=1,\dots,n\,,
			\end{equation*}
			which in particular gives $ F^{11}>\kappa \mathcal{F} $ and thus $ F^{aa}\lambda_a^2\geq F^{11}\lambda_1^2 \geq \kappa \mathcal{F}\lambda_1^2 $. We may assume $ F^{aa}\lambda_a\leq F^{aa}\lambda_a^2/(6NS) $ because if this were not true we would have $ \kappa \mathcal{F}\lambda_1^2<6NS\mathcal{F}\lambda_1 $		
			\[
			4\beta'\left(F^{aa}\tilde{v}_{\bar a a}-\partial_t \tilde{v}\right)\geq -12SF^{aa}\phi_{\bar a a}-C\mathcal{F}\geq-\frac{2F^{aa}\lambda_a^2}{N} -C\mathcal{F}\,,
			\]
			This last inequality and \eqref{eq_in_C2_5} finally give
			\[
			0\geq 2\kappa \frac{\lambda_1^2}{N}-C\,,
			\]
			as was to be shown.
		\end{itemize}
		The desired bound is valid at the maximum point $ x_0 $ of $ G $, and then also globally.
	\end{proof}		
	
	\begin{remark}
		As in the elliptic case treated in \cite{GZ}, this is the only step of the proof of our main results that uses the assumption that the metric $ g $ is hyperk\"ahler.
	\end{remark}

	\section{Gradient estimate}\label{Grad}
	The bound find in the previous section is well-suited for a standard interpolation argument which allows us to obtain directly a gradient bound and consequently, also a Laplacian bound. We thank the anonymous referee for pointing out this proof to us, which simplifies our previous one.
	\medskip

	\begin{proposition}\label{gradient}
		Suppose there is a uniform constant $C$ such that
		\[
		\|\phi\|_{C^0}\leq C\,, \qquad \|\Delta_g \phi \|_{C^0}\leq C\left( \|\nabla \phi \|_{C^0}+1 \right)\,,
		\]
		then there is a uniform bound
		\[
		\|\phi \|_{C^1} \leq C\,.
		\]
	\end{proposition}		
	\begin{proof}
		Interpolation theory (see \cite[section 6.8]{GT}) reveals that for any $ \epsilon>0 $ and  $ 0<\alpha<1 $ there is a constant $ C_\epsilon>0 $ such that
		\[
		\norma{\phi}_{C^1}\leq C_\epsilon \norma{\phi}_{C^0}+\epsilon\norma{\phi}_{C^{1,\alpha}}\leq C_\epsilon C+\epsilon \|\phi\|_{C^{1,\alpha}}\,.
		\]
		Choosing $ p=\frac{4n}{1-\alpha}>4n $ Morrey's inequality and elliptic $ L^p $-estimates for the Laplacian yield 
		\[
		\norma{\phi}_{C^{1,\alpha}}\leq C' \norma{\phi}_{W^{2,p}}\leq C''\left( \norma{\phi}_{L^p}+\norma{\Delta_g \phi}_{L^p} \right)\leq C''\left(\norma{\phi}_{C_0}+\norma{\Delta_g \phi}_{C^0}\right)
		\]
		for some constants $ C',C''>0 $ depending only on $ \alpha $.
		
		Putting everything together, we obtain
		\[
		\norma{\phi}_{C^1}\leq C_\epsilon C+\epsilon C''C\left( 2+\norma{\phi}_{C^1} \right)\,, 
		\]
		from which we can conclude by choosing $\epsilon<(C''C)^{-1} $.
	\end{proof}

\section{Higher order estimates and long-time existence}\label{higher}
Here we improve the Laplacian estimate to a H\"older estimate of the quaternionic Hessian of $\phi$. We do so by following an argument of Alesker \cite{Alesker (2013)} suitably adapted to our parabolic framework. By bootstrapping we then obtain estimates of any order on the solution of \eqref{eq_main} and thus also long-time existence.
\medskip

\begin{proposition}\label{higher order}
	For each $k>0$, there exists a uniform constant $C_{k}$ depending on the allowed data, $k$, $\|\nabla \phi \|_{C^0}$ and an upper bound for $\Delta_g \phi$ such that
	\begin{equation}\label{6..1}
		\|\nabla^{k}\varphi\|_{C^0}\leq C_{k}\,,
	\end{equation}
	where $\nabla$ is the Levi-Civita connection with respect to $g$. Moreover we have long-time existence for $\phi$, i.e. $T=\infty$.
\end{proposition}
\begin{proof}
	Assume \eqref{6..1} and suppose $T<\infty$. It follows form \eqref{up and low of phit} that there exists a uniform constant $C$ such that
	\begin{equation*}
		|\varphi|\leq T\sup_{M\times [0,T)}|\partial_{t}\varphi|\leq CT\,, \quad \textrm{on} \ M\times [0,T)\,.
	\end{equation*}
	By this, \eqref{6..1} and short-time existence, one can extend the flow to $[0, T+\epsilon_{0})$ for some $\epsilon_{0}>0$, which yields a contradiction. The interested reader can find more details about this standard discussion in the proof of \cite[Theorem 3.1]{Tos18} (see also in \cite{SW13,Wei16} and references therein). 
	
	We showed that it is enough to prove \eqref{6..1}.  And we claim that \eqref{6..1}, follows once we have proved a H\"older bound for $\mathrm{Hess}_{\mathbb{H}} \varphi$ of the form
	\begin{equation}\label{eq Holder}
		|{\rm Hess}_\H \phi |_{C^{0,\alpha}(M\times [\epsilon,T))} \leq C_\epsilon\,,
	\end{equation}
	where $\epsilon\in (0,T)$ and $C_\epsilon$ is a uniform constant depending only on the initial data and $\epsilon$. Indeed, given the H\"older bound \eqref{eq Holder} for the matrix $\mathrm{Hess}_{\mathbb{H}} \varphi$ and the second order estimate for $\varphi$, we can differentiate the flow \eqref{eq_main} and then bootstrap using the Schauder estimates in order to obtain the uniform bound
	\[
	|\nabla^k \phi |_{C^{0,\alpha}(M\times[\epsilon,T))}\leq C_{\epsilon,k}\,, \quad \text{for any }k>0\,,
	\]
	where $C_{\epsilon,k}$ depends on $\epsilon $ and $k$. But since by standard parabolic theory the solution $\phi$ is uniquely determined by the initial and background data, we also have a uniform bound
	\[
	|\nabla^k \phi |_{C^{0,\alpha}(M\times[0,\epsilon))}\leq C_{\epsilon,k}\,, \quad \text{for any }k>0\,.
	\]
	
	The estimate \eqref{eq Holder} is standard, we prove it as a separate Proposition below.
\end{proof}

\begin{proposition}\label{Holder}
	For each $\epsilon\in (0,T)$ there exists $\alpha\in (0,1)$ and a uniform constant $C_\epsilon>0$ depending only on the allowed data, $\epsilon$, $\|\partial_t \phi \|_{C^0}$, and an upper bound for $\Delta_g \phi$ such that
	\begin{equation*}
		|{\rm Hess}_\H \phi |_{C^{0,\alpha}(M\times [\epsilon,T))} \leq C_\epsilon\,.
	\end{equation*}
\end{proposition}
\begin{proof}
	The  proof is  classical in flavour and represents an adaptation of Alesker's $C^{2,\alpha}$ estimate for the quaternionic Monge-Amp\`ere equation obtained in \cite{Alesker (2013)} and inspired by the argument of B\l{}ocki \cite{Blo05}.
	
	Again the proof is local, since $M$ is locally flat. Let $\mathcal{O}\subset \mathbb{H}^{n}$ be an arbitrary open subset. For each $\alpha\in (0,1)$, on $\mathcal{O}_{T}:=\mathcal{O}\times [0,T)$, we define
	\[
	\big[\varphi\big]_{\alpha,(x,t)}:=\sup_{(y,s)\in \mathcal{O}_{T}\setminus (x,t)}\frac{|\varphi(y,s)-\varphi(x,t)|}{(|y-x|+\sqrt{|s-t|})^{\alpha}}\,,\quad \big[\varphi\big]_{\alpha,\mathcal{O}_{T}}:=\sup_{(x,t)\in\mathcal{O}_{T}}\big[\varphi\big]_{\alpha,(x,t)}\,.
	\]
	The metric $g$ can be locally represented by a potential $w$ on $\mathcal{O}$, possibly shrinking $\mathcal{O}$ if necessary, in other words $g=\mathrm{Hess}_{\mathbb{H}} w$. Let us denote $u=w+\varphi$ and $U=\mathrm{Hess}_{\mathbb{H}} u$. By concavity of $F$, and the mean value theorem, for all $(x,t_{1}), (y,t_{2})\in \mathcal{O}\times [0,T)$, we have
	\begin{equation}\label{osc}
		\begin{split}
			\Re \, F^{rs}(y,t_{2})(u_{\bar{r}s}(x,t_{1})&-u_{\bar{r}s}(y,t_{2})) \\
			&\geq  \partial_{t}\varphi(x,t_{1})-\partial_{t}\varphi(y,t_{2})+h(x)-h(y)\\
			&\geq  \partial_{t}u(x,t_{1})-\partial_{t}u(y,t_{2})-C\|x-y\|\,,
		\end{split}
	\end{equation}
	for some constant $C$ depending on $\|h\|_{C^{1}}$.
	
	At this point we recall the following algebraic lemma by Alesker \cite[Lemma 4.9]{Alesker (2013)}, which is analogous to  \cite{GT} for the real case and  \cite{Blo05,Siu} in the complex setting.
	
	\begin{lemma}
		Let $\lambda,\Lambda\in \mathbb{R}$ satisfy $0<\lambda<\Lambda<+\infty$. There exist a uniform constant $N$, unit vectors $\xi^{1},\dots, \xi^{N}\in \mathbb{H}^{n}$ and positive numbers $\lambda_{*}< \Lambda_{*}<+\infty$, depending only on $n, \lambda, \Lambda$ such that any hyperhermitian matrix $A \in \H^{n,n}$ with eigenvalues lying in the interval $[\lambda, \Lambda]$ can be written as
		\begin{equation*}
			A=\sum_{k=1}^{N}\beta_{k}(\xi^{k})^*\otimes\xi^{k}\,, \qquad \textrm{i.e. } \, A_{rs}=\sum_{k=1}^{N}\beta_{k}\bar \xi^k_{r}\xi^k_{s}\,,
		\end{equation*}
		for some $\beta_{k}\in [\lambda_{*}, \Lambda_{*}]$.
	\end{lemma}
	
	Applying the lemma to $A=(F^{rs}(U))$, immediately yields
	\begin{equation*}
		\begin{split}
			\Re\,	F^{rs}(U(y))(u_{\bar rs}(y)-u_{\bar rs}(x))
			=&\Re \sum_{k=1}^{N}\beta_{k}(y)\bar \xi^k_{r} \xi^{k}_s(u_{\bar r s}(y)-u_{\bar rs}(x))\\	=&\sum_{k=1}^{N}\beta_{k}(y)(\Delta_{\xi^{k}}u(y)-\Delta_{\xi^{k}}u(x))
		\end{split}
	\end{equation*}
	for some functions $\beta_{k}(y)\in [\lambda_{*}, \Lambda_{*}]$, where, for any unit vector $\xi\in \mathbb{H}^{n}$, we denoted by $\Delta_{\xi}$ the Laplacian on any translate of the quaternionic line spanned by $\xi$, i.e.
	\begin{equation*}
		\Re\, \textrm{tr}((\xi^*\otimes\xi)(u_{\bar{r}s}))
		=\Re\, \textrm{tr}(\xi(u_{\bar{r}s})\xi^*)
		=\Delta_{\xi}u\,.
	\end{equation*}
	Here we are using the well-known identity $\Re\, \mathrm{tr}(B_1B_2)=\Re \,\mathrm{tr}(B_2B_1)$ valid for any two quaternionic matrices $B_1,B_2$ for which the product is defined.
	
	For convenience, let us set $\beta_{0}(y)\equiv 1$ and
	$\Delta_{\xi^{0}}=-\partial_t.$
	Then, from \eqref{osc} we obtain
	\begin{equation}\label{sum1}
		\sum_{k=0}^{N}\beta_{k}\big(\Delta_{\xi^{k}}u(y,t_{2})
		-\Delta_{\xi^{k}}u(x,t_{1})\big)\leq C\|x-y\|\,.
	\end{equation}
	
	\begin{lemma}\label{subequation}
		For any $k=0,1,\cdots, N$, 
		\begin{equation*}
			\partial_{t}\Delta_{\xi^{k}}u\leq \Re\, F^{rs}\big(\Delta_{\xi^{k}}u_{\bar{r}s}\big)+\Delta_{\xi^{k}}h\,.
		\end{equation*}
	\end{lemma}
	\begin{proof} For $k=0$. Applying $\partial_{t}$ to \eqref{eq_main}, we get
		\begin{equation*}
			\partial_{t}\big(\partial_{t}u\big)=\Re\, F^{rs}\partial_{t}\big(u_{\bar{r}s}\big)
		\end{equation*}
		and the lemma follows.
		
		For other $k\geq 1$, write $\xi^{k}=(\xi^{k}_{1},\cdots, \xi^{k}_{n})$. Differentiating  \eqref{eq_main} along $\xi_{p}^{k}$ twice and taking sum over the index $p$, gives
		\begin{equation*}
			\begin{split}
				\partial_{t}\Delta_{\xi^{k}}u &=
				\Re\, F^{rs}\big(\Delta_{\xi^{k}}u_{\bar{r}s}\big)
				+\Re \sum_{p=1}^{n}F^{rs,tl}u_{\bar{r}s\xi_{p}^{k}}u_{\bar{t}l\xi_{p}^{k}}-\Delta_{\xi^{k}}h\\
				&\leq \Re \, F^{rs}\big(\Delta_{\xi^{k}}u_{\bar{r}s}\big) -\Delta_{\xi^{k}}h\,,
			\end{split}
		\end{equation*}
		by the concavity of $F$. Then the lemma follows.
	\end{proof}
	
	Fix $\hat{t}\in [\epsilon,T)$, and $r\in (0,1)$ such that $10r^{2}\leq \hat{t}$. Define 
	\[
	\begin{split}
		P_{r}&=\big\{(x,t)\in \mathcal{O}_{T}: \|x\|\leq r, \hat{t}-5r^{2}\leq t\leq \hat{t}-4r^{2}\big\}\,,\\
		Q_{r}&=\big\{(x,t)\in \mathcal{O}_{T}: \|x\|\leq r, \hat{t}-r^{2}\leq t\leq \hat{t}\big\}\,.
	\end{split}
	\]
	For every $k=0,1,\cdots, N$, let us denote
	\[
	M_{k,r}=\sup_{Q_{r}}\Delta_{\xi^{k}}u\,, \quad m_{k,r}=\inf_{Q_{r}}\Delta_{\xi^{k}}u\,, \quad \eta(r)=\sum_{k=1}^{N}(M_{k,r}-m_{k,r})\,.
	\]
	To prove Proposition \ref{Holder}, it suffices to find a constant $C$ (depending only on $\epsilon$), $r_{0}>0$ and $0<\delta<1$ such that
	\begin{equation*}
		\eta(r)\leq Cr^{\delta}\,, \quad \text{for all } r<r_{0}\,.
	\end{equation*}
	
	Let us define an operator $\mathcal{D}=\frac{1}{4}\Re\, F^{rs}(U)\partial_{\bar q^r} \partial_{q^s}$. Let $(a_{ij})\in \text{Sym} (4n,\R)$  be the realization of $(F^{rs}(U))$. Then we can rewrite $\mathcal{D}$ as
	\begin{equation}\label{non-div}
		\mathcal{D}=\sum_{s,t=1}^{4n}a_{st}D_{s}D_{t}\,,
	\end{equation}
	Since $F$ is uniformly elliptic on $\Gamma$, then $(a_{st})\in \mathrm{Sym}(4n,\R)$  satisfies the uniform elliptic estimate
	$\lambda\|\xi\|^{2}\leq \sum_{s,t}a_{st}\xi_{s}\xi_{t}\leq \Lambda \|\xi\|^{2}$ for some $0<\lambda<\Lambda<\infty$ and any $\xi\in \mathbb{R}^{4n}$.
	
	The following weak parabolic Harnack inequality is well-known.
	\begin{lemma}\cite[Theorem 7.37]{Lie}\label{weak harnack}
		If $v\in W_{2n+1}^{2,1}$ is a nonnegative function and satisfies
		\begin{equation*}
			-\frac{\partial v}{\partial t}+ \sum_{s,t}a_{st}D_{s}D_{t}v\leq h' \textrm{ on }Q_{4r}\,,
		\end{equation*}
		where $h'$ is a bounded function and the matrix $(a_{st})$ is as in \eqref{non-div}. Then there exist positive constants $C,p$ depending on $n,\lambda,\Lambda$ such that
		\begin{equation}\label{weak harnack 1}
			\frac{1}{r^{4n+2}}\left(\int_{P_{r}}v^{p}\right)^{\frac{1}{p}}\leq C\left(\inf_{B_{r}}v+r^{\frac{4n}{4n+1}}\|h'\|_{L^{2n+1}}\right).
		\end{equation}
	\end{lemma}
	For each $k=0,1,\cdots, N$, let us denote $v_{k}:=M_{k,2r}-\Delta_{\xi^{k}}u$. Then $v_{k}\in W_{2n+1}^{2,1}$ is a non-negative function and since $\Delta_{\xi^{k}}u_{\bar{r}s}=(\Delta_{\xi^{k}}u)_{\bar{r}s}$ on $\mathcal{O}_{T}$ it satisfies
	\[
	-\partial_{t}v_{k}+\Re\, F^{rs}(v_k)_{\bar rs} \leq h'
	\]
	for a bounded function $h'$.
	Then by Lemmas \ref{subequation} and \ref{weak harnack},
	\begin{equation}\label{weak harnack 2}
		\frac{1}{r^{4n+2}}\Big(\int_{P_{r}}(M_{k,2r}-\Delta_{\xi^{k}}u)^{p}\Big)^{\frac{1}{p}}\leq C\big(M_{k,2r}-M_{k,r}+r^{\frac{4n}{4n+1}}\big)\,,
	\end{equation}
	On the other hand, let $(x,t_{1}),(y,t_{2})\in Q_{2r}$, it then follows  from \eqref{sum1} that
	\begin{equation*}
		\beta_{k}\big(\Delta_{\xi^{k}}u(y,t_{2})-\Delta_{\xi^{k}}u(x,t_{1})\big)
		\leq  Cr+\sum_{\substack{ 0\leq \gamma\leq N \\ \gamma\neq k}}\beta_{\gamma}\big(\Delta_{\xi^{\gamma}}u(x,t_{1})
		-\Delta_{\xi^{\gamma}}u(y,t_{2})\big)\,.
	\end{equation*}
	For each $\epsilon>0$, pick a point $(x,t_{1})\in Q_{2r}$ such that $m_{k,2r}\leq \Delta_{\xi^{k}}u(x,t_{1})+\epsilon.$ As a consequence, after dividing the inequality above by $\beta_{k}$, we obtain
	\begin{equation*}
		\begin{split}
			\Delta_{\xi^{k}}u(y,t_{2})-m_{k,2r}
			\leq Cr+C\sum_{\substack{ 0\leq \gamma\leq N \\ \gamma\neq k}}(M_{\gamma,2r}-\Delta_{\xi^{\gamma}}u(y,t_{2}))\,,
		\end{split}
	\end{equation*}
	by arbitrariness of $\epsilon$. Integrating for $(y,t_{2})$ over $P_{r}$, and using the fundamental inequality $\|a+b\|_{p}\leq \|a\|_{p}+\|b\|_{p}$ for every $p>1$, yields
	\begin{equation}\label{weak harnack 3}
		\begin{split}
			\frac{1}{r^{4n+2}}&\Big(\int_{P_{r}}\big(\Delta_{\xi^{k}}u(y,t_{2})-m_{k,2r}\big)^{p}\Big)^{\frac{1}{p}}\\
			&\leq  \frac{C}{r^{4n+2}}\Big(\int_{P_{r}}\Big[r+\sum_{\substack{0\leq \gamma\leq N \\ \gamma\neq k}}(M_{\gamma,2r}-\Delta_{\xi^{\gamma}}u(y,t_{2}))\Big]^{p}\Big)^{\frac{1}{p}}\\
			&\leq  Cr+\frac{C}{r^{4n+2}}\sum_{\substack{ 0\leq\gamma\leq N \\ \gamma\neq k}}\Big(\int_{P_{r}}[M_{\gamma,2r}-\Delta_{\xi^{\gamma}}u(y,t_{2})]^{p}\Big)^{\frac{1}{p}}\\
			&\stackrel{\eqref{weak harnack 2}}{\leq}  C\sum_{\substack{ 0\leq \gamma\leq N \\ \gamma\neq k}}(M_{\gamma,2r}-M_{\gamma,r})+Cr^{\frac{4n}{4n+1}}\,,
		\end{split}
	\end{equation}
	where we have used the fact $0<r<1$ in the last inequality. In light of \eqref{weak harnack 2} and \eqref{weak harnack 3}, and again the triangle inequality $\|a+b\|_{p}\leq \|a\|_{p}+\|b\|_{p}$, we obtain
	\begin{equation*}
		\begin{split}
			M_{k,2r}-m_{k,2r} \leq\,\,& 
			\frac{C}{r^{4n+2}}\Big(\int_{P_{r}}(M_{k,2r}-\Delta_{\xi^{k}}u)^{p}\Big)^{\frac{1}{p}}\\
			&+\frac{C}{r^{4n+2}}\Big(\int_{P_{r}}(\Delta_{\xi^{k}}u-m_{k,2r})^{p}\Big)^{\frac{1}{p}}\\
			\leq\,\,&   C\sum_{\gamma=0}^{N}(M_{\gamma,2r}-M_{\gamma,r})+Cr^{\frac{4n}{4n+1}}\,.
		\end{split}
	\end{equation*}
	Summing over $k$ we deduce
	\[
	\eta(2r)\leq C\sum_{\gamma=0}^{N}(M_{\gamma,2r}-M_{\gamma,r})+Cr^{\frac{4n}{4n+1}}\,.
	\]
	By definition, $m_{\cdot,s}$ is non-increasing in $s$, whence
	\begin{equation*}
		\begin{split}
			\eta(2r)&\leq C\sum_{\gamma=0}^{N}\big((M_{\gamma,2r}-m_{\gamma,2r})
			-M_{\gamma,r}+m_{\gamma,r}\big)+Cr^{\frac{4n}{4n+1}} \\
			&=  C\big(\eta(2r)-\eta(r)\big)+Cr^{\frac{4n}{4n+1}}\,.
		\end{split}
	\end{equation*}
	Equivalently,
	\begin{equation*}
		\eta(r)\leq \Big(1-\frac{1}{C}\Big)\eta(2r)+Cr^{\frac{4n}{4n+1}}\,.
	\end{equation*}
	Applying a standard iteration technique (see \cite[Chapter 8]{GT} for more details), we finally infer that there exists a dimensional constant $\delta\in (0,1)$ such that $\eta(r)\leq Cr^{\delta}$ as we wanted to show.
	This completes the proof of Proposition \ref{Holder}. 
\end{proof}

\section{Convergence of the flow and proof of Theorems \ref{unbounded theorem} and \ref{bounded theorem}}\label{conv}

\subsection{Li-Yau type inequality}\hfill\\
\vskip-0.4em

Now we consider the following Li-Yau \cite{LY86} type equation
\begin{equation}\label{heat}
	(\mathcal{L}-\partial_{t})\psi=0\,, \quad \psi>0\,,
\end{equation}
where $\mathcal{L}=\frac{1}{4}\Re\, F^{ik}\partial_{\bar q_{i}}\partial_{q_{k}}$.

If $\Phi$ is a $C^2$ function and we let
\begin{equation*}
	\Phi_{k}:=\sum_{p=0}^{3}\Phi_{x_{p}^{k}}\, \bar{e}_{p}, \qquad
	\Phi_{\bar k}:=\sum_{p=0}^{3}e_{p}\Phi_{x_{p}^{k}}\,,
\end{equation*}
where $\Phi_{x_{p}^{k}}:=\frac{\partial\Phi}{\partial x_{p}^{k}}$,  and $\bar{e}_{p}$ denotes the quaternionic conjugate of the quaternionic unit $e_{p}$ for every $p$, then we can rewrite $\mathcal{L}$ as
\begin{equation*}
	\mathcal{L}\Phi=\frac{1}{4}\Re\, F^{ik}\partial_{\bar q_{i}}\partial_{q_{k}}\Phi= F^{ik}_{pq}\,\Phi_{x_{p}^{k}x_{q}^{i}},
\end{equation*}
where $F^{ik}_{pq}:= \frac{1}{4}\Re\, \{F^{ik}\bar{e}_{p}e_{q}\}$  for simplicity. This follows directly from the identity $\Re(ab)=\Re(ba)$ valid for any pair of quaternions $a,b\in \H$.

Let $B$ be a constant so large that $\psi=\partial_t \phi +B$ is a solution to \eqref{heat}. We consider the quantity
\[
H=t(|\partial v|^{2}-\alpha \partial_{t} v)\,, \qquad v=\log \psi\, ,
\]
where $\alpha\in (1,2)$ is a constant and
$$
|\partial v|^{2}=\frac{1}{4}\Re\, F^{jl}v_{j}v_{\bar{l}}=F^{jl}_{rs}\,v_{x_{r}^{l}}v_{x_{s}^{j}}\,.
$$

\begin{lemma}\label{Lemma heat1} There exists a constant $C>0$ such that
	\begin{equation}\label{heat1}
		(\mathcal{L}-\partial_{t})H\geq 
		\frac{t}{4n}\big(|\partial v|^{2}- \partial_{t} v\big)^{2}
		-2\langle\partial v,{\partial}H \rangle
		-\big(|\partial v|^{2}-\alpha \partial_{t} v \big)
		-tC|\partial v|^{2}-Ct\,,
	\end{equation}
	where $\langle\cdot,\cdot\rangle $ is the inner product defined by $\langle \partial f,\partial g \rangle=\frac14\Re \,F^{ik}f_{i}g_{\bar{k}}=F^{ik}_{pq}\,f_{x_{p}^{k}}g_{x_{q}^{i}}$ on real-valued $C^1$ functions.
\end{lemma}
\begin{proof}
	The proof is local. For each $z\in M$, we can find quaternionic coordinates $q_{1},\dots,q_{n}$ on a local chart around $z$. Plugging $\psi=e^{v}$ into \eqref{heat} we have
	\begin{equation}\label{easy fact}
		\mathcal{L}  v-\partial_{t} v=-|\partial v|^{2}\,,
	\end{equation}
	giving
	\begin{equation}\label{definition of H}
		H=-t\mathcal{L}  v-t(\alpha-1)\partial_{t} v\,,
	\end{equation}
	and thus also
	\begin{equation}\label{time derivatives for laplacian}
		t\partial_{t}\big(\mathcal{L}  v\big)=\frac{1}{t}H-\partial_{t} H-t(\alpha-1)\partial^{2}_{t}  v\,.
	\end{equation}
	By a straightforward computation we get
	\begin{equation}\label{derivate of H}
		\begin{split}
			-\partial_{t}H &
			=-\big(|\partial v|^{2}-\alpha \partial_{t} v \big)-2t \big\langle \partial v,{\partial}\partial_{t} v \big\rangle +t\alpha \partial^{2}_{t}  v-t \partial_{t}(F^{ik}_{pq})v_{x_{p}^{k}}v_{x_{q}^{i}}\,,\\  
			\mathcal{L}  H&=  t \mathcal{L}(|\partial v|^{2})-t \alpha\mathcal{L} (\partial_{t} v)\,.
		\end{split}
	\end{equation}
	First we deal with the term $\mathcal{L}(|\partial v|^2)$. For convenience, let us define 
	\begin{equation*}
		\mathcal{V}=F^{ik}_{pq}F^{jl}_{rs}v_{x_{r}^{l}x_{p}^{k}}v_{x_{s}^{j}x_{q}^{i}}\,, \qquad
		\mathcal{W}=F^{ik}_{pq}F^{jl}_{rs}v_{x_{r}^{l}x_{q}^{i}}v_{x_{s}^{j}x_{p}^{k}}\,.
	\end{equation*}
	By a direct calculation, we get
	\begin{equation*}
		\begin{split}
			\mathcal{L}(|\partial v|^{2})= \mathcal{V}&+\mathcal{W}+\mathcal{L} (F^{jl}_{rs})v_{x_{r}^{l}}v_{x_{s}^{j}}
			+F^{ik}_{pq}(F^{jl}_{rs})_{x_{p}^{k}}v_{x_{r}^{l}}v_{x_{s}^{j}x_{q}^{i}}\\
			&
			+F^{ik}_{pq}(F^{jl}_{rs})_{x_{p}^{k}}v_{x_{r}^{l}x_{q}^{i}}v_{x_{s}^{j}}+F^{ik}_{pq}(F^{jl}_{rs})_{x_{q}^{i}}v_{x_{r}^{l}x_{p}^{k}}v_{x_{s}^{j}}\\&
			+F^{ik}_{pq}(F^{jl}_{rs})_{x_{q}^{i}}v_{x_{r}^{l}x_{p}^{k}}v_{x_{s}^{j}}
			+F^{jl}_{rs}\mathcal{L}(v_{x_{r}^{l}})v_{x_{s}^{j}}
			+F^{jl}_{rs}v_{x_{r}^{l}}\mathcal{L}(v_{x_{s}^{j}})\,.
		\end{split}
	\end{equation*}
	Note that $\varphi$ has uniformly bounded $C^{k}$ norms for every $k>0$ by Proposition \ref{higher order}. Hence, analogously to the (almost) Hermitian case \cite{Chu2,Gill}, we deduce
	\begin{equation}\label{1 4}
		|\mathcal{L} (F^{jl}_{rs})v_{x_{r}^{l}}v_{x_{s}^{j}}|\leq C|\partial v|^{2}\,.
	\end{equation}
	For each $0<\epsilon<1$, we have that
	\begin{equation}\label{53}
		\begin{split}
			\big|F^{ik}_{pq}(F^{jl}_{rs})_{x_{p}^{k}}v_{x_{r}^{l}}v_{x_{s}^{j}x_{q}^{i}}|
			+ &  |F^{ik}_{pq}(F^{jl}_{rs})_{x_{p}^{k}}v_{x_{r}^{l}x_{q}^{i}}v_{x_{s}^{j}}\big|+|F^{ik}_{pq}(F^{jl}_{rs})_{x_{q}^{i}}v_{x_{r}^{l}x_{p}^{k}}v_{x_{s}^{j}}|\\
			+&|F^{ik}_{pq}(F^{jl}_{rs})_{x_{q}^{i}}v_{x_{r}^{l}x_{p}^{k}}v_{x_{s}^{j}}|
			\leq \frac{C}{\epsilon}|\partial v|^{2}+2\epsilon \mathcal{W}
			+2\epsilon \mathcal{V}\,.
		\end{split}
	\end{equation}
	Observe that $(\mathcal{L} v)_{x_{s}^{j}}-\mathcal{L}(v_{x_{s}^{j}})= (F^{ik}_{pq}v_{x_{p}^{k}x_{q}^{i}})_{x_{s}^{j}}-F^{ik}_{pq}v_{x_{s}^{j}x_{p}^{k}x_{q}^{i}}=(F^{ik}_{pq})_{x_{s}^{j}}v_{x_{p}^{k}x_{q}^{i}}. $
	It follows that
	\begin{equation}\label{three order}
		\begin{split}
			F^{jl}_{rs}\mathcal{L}(v_{x_{r}^{l}})v_{x_{s}^{j}}
			+&F^{jl}_{rs}v_{x_{r}^{l}}\mathcal{L}(v_{x_{s}^{j}})-2\big\langle \partial v, \partial \mathcal{L}  v\big\rangle\\
			=&F^{jl}_{rs}v_{x_{s}^{j}}(\mathcal{L}(v_{x_{r}^{l}})-(\mathcal{L} v)_{x_{r}^{l}})
			+F^{jl}_{rs}v_{x_{r}^{l}}(\mathcal{L}(v_{x_{s}^{j}})-(\mathcal{L} v)_{x_{s}^{j}})\\
			=&-F^{jl}_{rs}v_{x_{s}^{j}}(F^{ik}_{pq})_{x_{r}^{l}}v_{x_{p}^{k}x_{q}^{i}}
			-F^{jl}_{rs}v_{x_{r}^{l}}(F^{ik}_{pq})_{x_{s}^{j}}v_{x_{p}^{k}x_{q}^{i}}\\
			\geq& -\frac{C}{\epsilon}|\partial v|^{2}-\epsilon \mathcal{V}-\epsilon \mathcal{W}\,.
		\end{split}
	\end{equation}
	On the other hand,\begin{equation}\label{4 9 term1}
		\begin{split}
			2t\big\langle \partial v, \partial \mathcal{L}  v\big\rangle\stackrel{\eqref{definition of H}}{=}&
			-2\big\langle \partial v, \partial H\big\rangle
			-2t(\alpha-1)\langle \partial v, \partial \partial_{t} v\rangle\\
			\stackrel{\eqref{derivate of H}}{=}&
			-2\big\langle \partial v, \partial H\big\rangle
			-(\alpha-1)\partial_{t}H+\frac{\alpha-1}{t}H-t\alpha(\alpha-1) \partial^{2}_{t}  v\\
			&-t(\alpha-1)\partial_{t}(F^{ik}_{pq})v_{x_{p}^{k}}v_{x_{q}^{i}}\\
			\geq &
			-2\big\langle \partial v, \partial H\big\rangle
			-(\alpha-1)\partial_{t}H+\frac{\alpha-1}{t}H-t\alpha(\alpha-1) \partial^{2}_{t}  v-Ct|\partial v|^{2}\,.
		\end{split}
	\end{equation}
	It follows from \eqref{three order} and \eqref{4 9 term1} that
	\begin{equation}\label{56}
		\begin{split}
			t\Big(F^{jl}_{rs}&\mathcal{L}(v_{x_{r}^{l}})v_{x_{s}^{j}}
			+F^{jl}_{rs}v_{x_{r}^{l}}\mathcal{L}(v_{x_{s}^{j}})\Big)\\
			\geq &
			-2\big\langle \partial v, \partial H\big\rangle
			-(\alpha-1)\partial_{t}H+
			\frac{\alpha-1}{t}H-t\alpha(\alpha-1) \partial^{2}_{t}  v \\
			&{} - Ct|\partial v|^{2}-\frac{Ct}{\epsilon}|\partial v|^{2}-t\epsilon \mathcal{V}-t\epsilon \mathcal{W}\,.
		\end{split}
	\end{equation}
	Now, we treat the second term of $\mathcal{L}H$ in \eqref{derivate of H}. Using the Cauchy-Schwarz inequality, at $z$, we deduce
	\begin{equation}\label{10 term}
		\begin{split}
			-t\alpha  \mathcal{L} (\partial_{t} v)
			=& -t\alpha  \partial_{t}(\mathcal{L}  v)
			+t\alpha \partial_{t}(F^{ik}_{pq})v_{x_{p}^{k}x_{q}^{i}}\\
			\stackrel{\eqref{time derivatives for laplacian}}{=} &  -\frac{\alpha}{t}H+\alpha\partial_{t}H+t\alpha(\alpha-1)\partial^{2}_{t}  v
			+t\alpha\partial_{t}(F^{ik}_{pq})v_{x_{p}^{k}x_{q}^{i}}\\
			\geq &-\frac{\alpha}{t}H+\alpha\partial_{t}H+t\alpha(\alpha-1)\partial^{2}_{t}  v
			-\frac{Ct}{\epsilon}-t\epsilon \mathcal{V}\,,
		\end{split}
	\end{equation}
	where in the last inequality we have used the fact that $-CF^{ik}_{pq}\leq \partial_{t}(F^{ik}_{pq})\leq CF^{ik}_{pq}$ for a uniform constant $C$, which is implied by Proposition \ref{higher order}.
	
	Plugging  \eqref{1 4}, \eqref{53}, \eqref{56} and  \eqref{10 term} into  \eqref{derivate of H}, we get
	\[
	\begin{split}
		\mathcal{L} H\geq&\, t\mathcal{W}+t\mathcal{V}-Ct|\partial v|^{2}
		-t\Big(\frac{C}{\epsilon}|\partial v|^{2}+2\epsilon \mathcal{V}+2\epsilon \mathcal{W}\Big) -2\langle \partial v, \partial H\rangle
		-(\alpha-1)\partial_{t}H\\
		&
		+\frac{\alpha-1}{t}H
		-t\alpha(\alpha-1) \partial^{2}_{t}  v
		-Ct|\partial v|^{2}-\frac{Ct}{\epsilon}|\partial v|^{2}-t\epsilon (\mathcal{V}+\mathcal{W})\\
		&
		-\frac{\alpha}{t}H+\alpha\partial_{t}H+t\alpha(\alpha-1)\partial^{2}_{t}  v
		-\frac{Ct}{\epsilon}-t\epsilon \mathcal{V}\\
		\geq &\,t(1-4\epsilon)\mathcal{V}
		+t(1-3\epsilon)\mathcal{W}
		-\frac{4Ct}{\epsilon}|\partial v|^{2}
		+\partial_{t}H -\frac{1}{t}H
		-2\langle \partial v, \partial H\rangle -\frac{Ct}{\epsilon}\,.
	\end{split}
	\]
	Thus, if we choose $\frac{1}{16}\leq\epsilon\leq \frac{1}{8}$, 
	\begin{equation}\label{heat 2}
		(\mathcal{L}-\partial_{t})H\geq  \frac{t}{2}\mathcal{V}-Ct|\partial v|^{2} -\big(|\partial v|^{2}-\alpha \partial_{t} v \big)
		-2\langle \partial v, \partial H\rangle -Ct\,.
	\end{equation}
	Applying the aritheoremetic-geometric mean inequality, and by \eqref{easy fact},
	\begin{equation*}
		\mathcal{V}
		\geq \frac{1}{n}(\mathcal{L} v)^{2}
		=\frac{1}{n}\big(\partial_{t} v-|\partial v|^{2}\big)^{2}\,.
	\end{equation*}
	Plugging it into  \eqref{heat 2}, we infer that
	\begin{equation*}
		(\mathcal{L}-\partial_{t})H\geq  \frac{t}{2n}\big(\partial_{t} v-|\partial v|^{2}\big)^{2}-Ct|\partial v|^{2}  -\big(|\partial v|^{2}-\alpha \partial_{t} v \big)-2\langle \partial v, \partial H\rangle -Ct\,.
	\end{equation*}
	By the arbitrariness of $z$, this proves  \eqref{heat1}.
\end{proof}

Using the parabolic maximum principle, we can prove the following lemma.
\begin{lemma}\label{Lemma heat2} On $M\times (0,T)$, we have
	\begin{equation*}\label{}
		|\partial v|^{2}-\alpha \partial_{t} v\leq \frac{8n\alpha^{2}}{t}+\sqrt{8n\alpha^{2}\Big(C+\frac{nC^{2}\alpha^{2}}{2(\alpha-1)^{2}}\Big)}\,.
	\end{equation*}
\end{lemma}
\begin{proof}
	Let us fix an arbitrary time $t_{0}\in (0,T)$. Suppose $H(x,t)$ (as in \eqref{definition of H}) achieves its maximum at the point $(\hat{q},\hat{t})\in M\times [0,t_{0}]$. We may assume $\hat{t}>0$, otherwise $|\partial v|^{2}-\alpha \partial_{t} v\leq 0$ on $M\times [0,t_{0}]$ and we are done. It follows that
	$$H(\hat{q},\hat{t})\geq H(\hat{q},0)=0\,.$$
	Using the maximum principle at $(\hat{q},\hat{t})$, we deduce $(\mathcal{L}-\partial_{t})H\leq 0$ and $\partial H=0$. Substituting this into  \eqref{heat1} yields
	\begin{equation}\label{max principle}
		\begin{split}
			\frac{\hat{t}^{2}}{4n}\big(|\partial v|^{2}-\partial_{t} v\big)^{2}-C\hat{t}^{2}|\partial v|^{2}
			-H \leq C\hat{t}^{2}\,.
		\end{split}
	\end{equation}
	Notice that at $(\hat{q},\hat{t})$,
	\begin{equation}\label{first term}
		\begin{split}
			\hat{t}^{2}\big(|\partial v|^{2}-\partial_{t} v\big)^{2}&= \frac{\hat{t}^{2}}{\alpha^{2}}\big(|\partial v|^{2}-\alpha \partial_{t} v+(\alpha-1)|\partial v|^{2}\big)^{2}\\
			&= \frac{H^{2}}{\alpha^{2}}+\Big(\frac{\alpha-1}{\alpha}\Big)^{2}\hat{t}^{2}|\partial v|^{4}+\frac{2(\alpha-1)\hat{t}H}{\alpha^{2}}|\partial v|^{2}\\
			&\geq \frac{H^{2}}{\alpha^{2}}+\Big(\frac{\alpha-1}{\alpha}\Big)^{2}\hat{t}^{2}|\partial v|^{4}\,,
		\end{split}
	\end{equation}
	where we have used the fact that $H$ is nonnegative at $(\hat{q},\hat{t})$.
	Using the elementary inequality $ax^{2}+bx\geq -\frac{b^{2}}{4a}$, we get
	\begin{equation}\label{second term}
		\frac{1}{4n}\Big(\frac{\alpha-1}{\alpha}\Big)^{2}\hat{t}^{2}|\partial v|^{4}
		-\hat{t}^{2}C|\partial v|^{2}
		\geq -\frac{nC^{2}\alpha^{2}}{2(\alpha-1)^{2}}\hat{t}^{2}\,.
	\end{equation}
	Plugging \eqref{first term} and \eqref{second term} into  \eqref{max principle} gives
	\begin{equation*}
		\frac{H^{2}}{4n\alpha^{2}}\leq H+ C\hat{t}^{2}+\frac{nC^{2}\alpha^{2}}{2(\alpha-1)^{2}}\hat{t}^{2}\,;
	\end{equation*}
	from which we can deduce
	\begin{equation*}
		H(\hat{q},\hat{t})\leq 8n\alpha^{2}+\sqrt{8n\alpha^{2}\Big(C+\frac{nC^{2}\alpha^{2}}{2(\alpha-1)^{2}}\Big)}\hat{t}\,.
	\end{equation*}
	Hence, at each point $q\in M$,
	\begin{equation*}
		\begin{split}
			H(q,t_{0})\leq & H(\hat{q},\hat{t})
			\leq 8n\alpha^{2}+\sqrt{8n\alpha^{2}\Big(C+\frac{nC^{2}\alpha^{2}}{2(\alpha-1)^{2}}\Big)}t_{0}\,.
		\end{split}
	\end{equation*}
	Consequently, at $(q,t_{0})$,
	\[
	|\partial v|^{2}-\alpha \partial_{t} v\leq \frac{8n\alpha^{2}}{t_{0}}+\sqrt{8n\alpha^{2}\Big(C+\frac{nC^{2}\alpha^{2}}{2(\alpha-1)^{2}}\Big)}\,.
	\]
	Then the lemma follows by arbitrariness of $t_{0}$.
\end{proof}

\medskip

\subsection{Parabolic Harnack inequality} \hfill\\
\vskip-0.4em

Let $\psi=\partial_{t}\varphi+B$ for a large constant $B$ such that $\psi>0$ on $M$. By \eqref{heat equation} we know
\begin{equation}\label{box}
	\mathcal{L}\psi-\partial_{t}\psi=0\,.
\end{equation}
With the results of the previous subsection we can prove the following useful parabolic Harnack inequality:
\begin{proposition}
	Let $0<t_{1}<t_{2}<T$. Then there exist constants $C_{i}$ $(i=1,2,3)$ depending only on $(M,I,J,K)$, $\Omega$ and $f$ such that
	\begin{equation}\label{harnack}
		\sup_{M}\psi(\cdot,t_{1})\leq \inf_{ M}\psi(\cdot,t_{2})\left(\frac{t_{2}}{t_{1}}\right)^{C_{1}}
		\exp\left(\frac{C_{2}}{t_{2}-t_{1}}+C_{3}(t_{2}-t_{1})\right).
	\end{equation}
\end{proposition}
\begin{proof}
	With Lemmas \ref{Lemma heat1} and \ref{Lemma heat2}, we can apply the procedure of \cite{Chu2,Gill} verbatim.
\end{proof}

\medskip

\subsection{Convergence of the parabolic flow}

\begin{proposition}\label{Convergence}
	Suppose $T=\infty $, $ \osc_M \phi(\cdot,t)\leq C$ and $\|\nabla^k \phi \|_{C^0}\leq C$ for any $k>0$, where $C>0$ is a uniform constant. Then the normalization $\tilde{\phi}$ converges in $C^\infty$ topology to a smooth function $\tilde{\phi}_\infty$ that satisfies
	\[
	F(A[\tilde{\phi}_\infty])=h+b\,,
	\]
	for some constant $b\in \R$.
\end{proposition}
\begin{proof}
	Set $\psi=\partial_{t}\varphi+B$ for a large constant $B$ such that $\psi>0$. For each $m\in\mathbb{N}$, we define
	\[
	\check{\psi}_{m}(x,t):=\sup_{M}\psi(\cdot,m-1)-\psi(x,m-1+t)\,;
	\]
	\[
	\hat{\psi}_{m}(x,t):=\psi(x,m-1+t)-\inf_{M}\psi(\cdot,m-1)\,.
	\]
	It is straightforward to verify that
	\[
	(\partial_{t}-\mathcal{L})\psi=(\partial_{t}-\mathcal{L})\hat{\psi}_{m}=(\partial_{t}-\mathcal{L})\check{\psi}_{m}=0\,.
	\]
	Applying the parabolic Harnack inequality \eqref{harnack}, this yields
	\begin{equation*}
		\sup_{M}\hat{\psi}_{m}(\cdot,t_{1})\leq C\inf_{ M}\hat{\psi}_{m}(\cdot,t_{2})\,,
		\qquad
		\sup_{M}\check{\psi}_{m}(\cdot,t_{1})\leq C\inf_{ M}\check{\psi}_{m}(\cdot,t_{2})\,.
	\end{equation*}
	Choosing $t_{1}=\frac{1}{2}$, $t_{2}=1$ we get
	\begin{equation}\label{7.7}
		\begin{split}
			\sup_{M}\psi\left(\cdot,m-\frac{1}{2}\right)-\inf_{ M}\psi(\cdot,m-1)&\leq C\left(\inf_{M}\psi(\cdot,m)-\inf_{ M}\psi(\cdot,m-1)\right),\\
			\sup_{M}\psi(\cdot,m-1)-\inf_{ M}\psi\left(\cdot,m-\frac{1}{2}\right)&\leq C\left(\sup_{ M}\psi(\cdot,m-1)-\sup_{M}\psi(\cdot,m)\right).
		\end{split}
	\end{equation}
	In light of \eqref{7.7}, if we set
	\[
	\theta(t)=\sup_{M}\psi(\cdot,t)-\inf_{M}\psi(\cdot,t)
	\]
	for the oscillation, then we have
	\begin{equation*}
		\theta(m-1)+\theta\left(m-\frac{1}{2}\right)\leq C\big(\theta(m-1)-\theta(m)\big)\,,
	\end{equation*}
	which implies that $\theta(m)\leq e^{-\delta}\theta(m-1)$, where $\delta:=-\log(1-\frac{1}{C})>0$, and by induction
	\[
	\theta(t)\leq Ce^{-\delta t}\,.
	\]
	Since we have $\int_{M}\partial_{t}\tilde{\varphi}=0$, by the mean value theorem, there exists a point $x_{t}\in M$ such that $\partial_{t}\tilde{\varphi}(x_{t}, t)=0$. Therefore,
	\[
	\begin{split}
		\big|\partial_{t}\tilde{\varphi}(x,t)\big|&=\big|\partial_{t}\tilde{\varphi}(x,t)-\partial_{t}\tilde{\varphi}(x_{t},t)\big|
		\leq \osc_{M}\partial_{t}\tilde{\varphi}(\cdot,t)\\
		&= \osc_{M}\partial_{t}\varphi(\cdot,t)=\theta(t) \leq  Ce^{-\delta t}\,,  
	\end{split}
	\]
	which yields that $\tilde{\varphi}+\frac{C}{\delta}e^{-\delta t}$ (resp. $\tilde{\varphi}-\frac{C}{\delta}e^{-\delta t}$) is non-increasing (resp. non-decreasing) with respect to $t$. It then follows from the uniform bounds on $\phi$ that $\tilde{\phi}$ is uniformly bounded in $C^\infty$ topology, therefore there is a sequence of times $t_j\to \infty$ such that $\tilde \phi (\cdot,t_j)$ converges smoothly to some smooth function $\tilde{\phi}_\infty$ and it is fairly standard to show that actually $\lim_{t\rightarrow\infty}\tilde{\varphi}=\tilde{\phi}_\infty$ in the $C^{\infty}$ topology.
	
	Finally, the limiting function $\tilde{\phi}_\infty$ satisfies
	\[
	0=\lim_{t\to \infty} \partial_t \tilde{\phi}(\cdot,t)=\lim_{t\to \infty} \left( F(A[\tilde{\phi}])-h-\frac{\int_M \partial_t \phi\, \Omega_0^n\wedge \bar \Omega_0^n}{\int_M \Omega_0^n\wedge \bar \Omega_0^n} \right)=F(A[\tilde{\phi}_\infty])-h-b\,,
	\]
	where we set
	\[
	b=\lim_{t\to \infty}\frac{\int_M \partial_t \phi\, \Omega_0^n\wedge \bar \Omega_0^n}{\int_M \Omega_0^n\wedge \bar \Omega_0^n}\,. \qedhere
	\]
\end{proof}

\section{Proof of Theorems \ref{unbounded theorem}--\ref{theorem 5}}\label{final}
We are ready to complete the proofs of Theorems \ref{unbounded theorem} and \ref{bounded theorem}, from which we will infer Theorems \ref{theorem Hessian}, \ref{theorem n-1 qpsh flow} and \ref{theorem 5}.

\begin{proof}[Proof of Theorem $\ref{unbounded theorem}$]
	Let $ (M,I,J,K,g) $ be a compact flat hyperk\"ahler manifold, $ \phi,\tilde\phi \colon M \to \R $ be the solution to \eqref{eq_main} and its normalization (defined in \eqref{tilde varphi}). The initial datum $\phi_0$ is assumed $\Gamma$-admissible and, since $f$ is unbounded, every $\Gamma$-admissible function is automatically a parabolic $\mathcal{C}$-subsolution. Hence we may apply Proposition \ref{c0 estimate} and deduce $\osc_M \phi(\cdot,t)\leq C$ and $\|\tilde \phi \|_{C^0}\leq C$. This bounds allow to obtain from Propositions \ref{prop_C2-bound} and \ref{gradient} a uniform constant $C$ such that $\Delta_g \phi \leq C$. Applying now Proposition \ref{higher order} we infer long-time existence of $\phi$ and uniform bounds on its derivatives of any order. Finally, Proposition \ref{Convergence} yields smooth convergence of the normalization $\tilde{\phi}$ to some function $\tilde{\phi}_\infty$ which is a solution of \eqref{elliptic function}, i.e.
	\[
	F(A[\tilde{\phi}_\infty])=h+b
	\]
	for a suitable constant $b\in \R$.
\end{proof}

\begin{proof}[Proof of Theorem $\ref{bounded theorem}$]
	The proof is quite similar to the one of Theorem \ref{unbounded theorem}. Indeed, suppose $f$ is bounded on $\Gamma$ and assume that it satisfies either one of the two conditions expressed in the statement of Theorem \ref{bounded theorem}, we are still able to apply Proposition \ref{c0 estimate} and deduce $\osc_M \phi(\cdot,t)\leq C$ and $\|\tilde \phi \|_{C^0}\leq C$.  Now we can employ the arguments in the proof of Theorem \ref{unbounded theorem} to complete the proof.
\end{proof}


Now we prove Theorem \ref{theorem Hessian} and Theorem \ref{theorem n-1 qpsh flow} as applications of Theorem \ref{unbounded theorem}.
\begin{proof}[Proof of Theorem $ \ref{theorem Hessian} $]
	The result follows as a simple application of Theorem \ref{unbounded theorem} once we choose $ f=\log \sigma_k $ defined over the cone
	\[
	\Gamma=\Gamma_k:=\{ \lambda \in \R^n \mid \sigma_1(\lambda),\dots,\sigma_k(\lambda)>0 \}\,,
	\]
	where $ \sigma_r $ is the $ r $-th elementary symmetric function
	\[\sigma_r(\lambda)=\sum_{1\leq i_1<\dots<i_r\leq n} \lambda_{i_1}\cdots \lambda_{i_r}\,, \quad \text{ for all }\lambda=(\lambda_1,\dots,\lambda_n)\in \R^n\,.\]
	Indeed, on a locally flat hyperhermitian manifold a function $u$ of class  $C^2$  lies in $ {\rm QSH}_{k}(M,\Omega) $ if and only if it is $\Gamma_k$-admissible. The function $f$ satisfies our structural assumptions C1--C3 (see e.g. \cite{Spruck}) and it is straightforward to check that it is unbounded over $\Gamma_k$. Finally, with this setup, the quaternionic Hessian flow \eqref{qh} becomes $\partial_t \phi=f(\lambda(A[\phi]))-H,$ as desired.
\end{proof}

\begin{proof}[Proof of Theorem $\ref{theorem n-1 qpsh flow}$]
	Define
	\[
	f = \log\sigma_{n}(T)\,, \qquad \Gamma = T^{-1}(\Gamma_{n})\,,
	\]
	where $T\colon \R^n \to \R^n$ is the linear map defined by
	\[
	T(\lambda) = \big(T(\lambda)_{1},\ldots,T(\lambda)_{n}\big)\,, \qquad
	T(\lambda)_{k} = \frac{1}{n-1}\sum_{i\neq k}\lambda_{i}\,, \qquad \text{for every } \lambda\in\mathbb{R}^{n}\,.
	\]
	An easy verification shows that assumptions C1--C3 are satisfied and that $f$ is unbounded over $\Gamma$. Setting
	\[
	\Omega: =\Re \left(g^{\bar j s}(\Omega_{1})_{\bar j s}\right)\Omega_{0}-(n-1)\Omega_{1}\,,
	\]
	one can easily see that $u\in C^2(M,\R)$ lies in ${\rm QPSH}_{n-1}(M,\Omega_1,\Omega_0)$ if and only if $\lambda(A[u])\in \Gamma $, where $A[u]=g^{\bar j r}(\Omega_{\bar j s}+u_{\bar j s})$. We can then rewrite the $(n-1)$-quaternionic plurisubharmonic flow \eqref{n-1 qpsh flow} as $ \partial_t \phi =f(\lambda(A[\phi]))-H $ and apply Theorem \ref{unbounded theorem} to conclude.
\end{proof}

Finally, we conclude the paper with the proof of Theorem \ref{theorem 5}.

\begin{proof}[Proof of Theorem $\ref{theorem 5}$]
	Let $\underline \phi $ be an elliptic $\mathcal{C}$-subsolution of the equation
	\[F(A[\phi])=h\,,
	\]
	which we have shown that can be seen as a time-independent parabolic $\mathcal{C}$-subsolution of our flow \eqref{eq_main}. Consider flow \eqref{eq_main} with a $\Gamma$-admissible initial datum $\phi_0$, then condition \eqref{case-1} of Theorem \ref{bounded theorem} is trivially verified, and this concludes the proof.
\end{proof}

\bibliography{ijmsample}

\bibliographystyle{ijmart}

\end{document}